\documentclass[11pt,reqno,bold]{paper}
\usepackage[letterpaper,left=25mm,right=25mm,top=25mm,bottom=25mm,headsep=5mm,footskip=7mm]{geometry}

\usepackage{amsmath}
\usepackage{amssymb}
\usepackage{amsfonts}
\usepackage{amsthm}
\usepackage{mathrsfs}

\usepackage{enumerate}
\usepackage{color}
\definecolor{darkblue}{rgb}{0.0,0.0,0.4}
\usepackage[pdfauthor={Tsogtgerel Gantumur},pdftitle={prescribed scalar curvature mean curvature problem},pdfsubject={Conformal geometry},colorlinks,citecolor=darkblue,linkcolor=darkblue,urlcolor=darkblue]{hyperref}

\usepackage[backend=bibtex,style=numeric,maxcitenames=6,maxbibnames=6,doi=false,url=false,isbn=false]{biblatex}
\DeclareNameAlias{sortname}{first-last}
\bibliography{prescribed}

\newtheorem{theorem}{Theorem}[section]
\newtheorem{proposition}[theorem]{Proposition}
\newtheorem{lemma}[theorem]{Lemma}
\newtheorem{corollary}[theorem]{Corollary}
\theoremstyle{remark}

\newcommand{\R}{{\mathbb{R}}}
\newcommand\eps{\varepsilon}

\newcommand{\Tr}{{\gamma}}
\newcommand{\Yam}{\mathscr{Y}}

\newcommand\supp{\mathrm{supp}}
\newcommand\sig{\mathrm{sign}}
\newcommand\vol{\mathrm{vol}}

\title{A prescribed scalar and boundary mean curvature problem on compact manifolds with boundary}
\author{Vladmir Sicca and Gantumur Tsogtgerel}
\institution{McGill University}
\date{\today}                                           

\begin{document}

\maketitle

\begin{abstract}
We consider the problem of finding a metric in a given conformal class with prescribed nonpositive scalar curvature and nonpositive boundary mean curvature
on a compact manifold with boundary, and establish a necessary and sufficient condition in terms of a conformal invariant that measures the zero set of the target curvatures.
\end{abstract}

\section{Introduction}
\label{sec:main}

Let $(M,g)$ be a compact Riemannian manifold with boundary.
Suppose that $R'$ is a nonpositive function in $M$, and $H'$ is a nonpositive function on $\partial M$.
Then we ask the question: Can we conformally transform the metric $g$ so that the resulting metric has
the scalar curvature equal to $R'$, and the boundary mean curvature equal to $H'$?
Setting the dimension of $M$ to be three for the sake of this introduction, this is equivalent to solubility of the equation 
\begin{equation}\label{e:prob-eq}
\begin{split}
\begin{cases}
    -8\Delta u+R u&=R' u^{5}\qquad\textrm{in}\,\, M ,\\
2\partial_\nu u+Hu&=H'u^{3}\qquad\textrm{on}\,\, \partial M ,
\end{cases}
\end{split}
\end{equation}
for a positive function $u$,
where $R$ and $H$ are the scalar curvature and the boundary mean curvature of $g$, respectively, and 
$\partial_\nu$ is the outward normal derivative.
Apart from its intrinsic importance, we are led to this problem by the study of the Einstein constraint equations.
More specifically, in the negative Yamabe case, 
solubility of an important class of Lichnerowicz equations on compact manifolds with boundary
is shown to be equivalent to a certain instance of \eqref{e:prob-eq}, cf. \cite{HT13}.

The analogue of this problem on closed manifolds was first studied by Rauzy in \cite{Ra95a},
and more recently, the analysis has been extended to asymptotically flat manifolds by Dilts and Maxwell in \cite{DM15},
and to asymptotically hyperbolic manifolds by Gicquaid in \cite{gicq2019}.
Restricting ourselves to the boundaryless case,
those works show that a generalization of the Yamabe invariant to subsets of the manifold plays an important role,
in that \eqref{e:prob-eq} has a positive solution if and only if 
this generalized Yamabe invariant for the zero set of $R'$ is positive.
The positivity of the Yamabe invariant is in a certain sense a measure of the smallness of the subset, as observed for asymptotically Euclidean manifolds in Lemma 3.15 of \cite{DM15}.

In this work, we adapt these ideas -- particularly the ones in \cite{DM15} -- to the framework of manifolds with boundary. To do so we extend the Yamabe invariant of manifolds with boundary as defined by Escobar to
pairs $(\Omega,\Sigma)$, where $\Omega\subset M$ and $\Sigma\subset \partial M$, which we call the {\em relative Yamabe invariant}, by looking at test functions supported only on $\Omega\cup \Sigma$ and develop some basic properties of the invariant. In particular, we relate this invariant to a relative version of the first eigenvalue of the Laplacian of the pair $(\Omega,\Sigma)$, allowing us to use results related to the linear version of problem (\ref{e:prob-eq}). In the end, we can prove that, in the Yamabe negative case, \eqref{e:prob-eq} has a positive solution if and only if the relative Yamabe invariant of the pair $(\{R'=0\},\{H'=0\})$ is positive.

\subsection{Outline of the paper}

The paper is structured as follows. In section \ref{sec:yamabe} we take the necessary steps to define the relative Yamabe invariant in the framework of manifolds with boundary and establish its relation to the subcritical problem corresponding to (\ref{e:prob-eq}). In section \ref{sec:eigenvalue} we develop the theory of the linearized problem by defining the relative eigenvalue of the Laplacian and showing that it can be used in the study of the relative Yamabe invariant since both have the same sign. In section 4 we bring both ideas together in using variational techniques to solve the prescribed curvature problem for Yamabe negative manifolds with boundary when $R'\leq 0$ and $H'\leq 0$. Finally, in section \ref{sec:examples_and_consequences} we put our results in context with some immediate consequences and examples.

\section{The relative Yamabe invariant}
\label{sec:yamabe}

Let ${M}$ be a smooth, connected, compact manifold with boundary and dimension $n\geq 3$.
Assume that ${M}$ is equipped with a Riemannian metric $g\in W^{s,p}$,
where $sp>n$ and $s\geq 1$.
We denote by $R\in W^{s-2,p}({M})$ the scalar curvature of $(M,g)$, and by $H\in W^{s-1-\frac1p,p}(\partial{M})$ the mean extrinsic curvature of the boundary $\partial{M}$,
with respect to the outer normal.
Let $\Omega\subset M$ and $\Sigma\subset\partial M$ be (relatively) measurable sets,
and consider the functional $E:W^{1,2}({M})\to\R$ defined by
\begin{equation}
E(\varphi)=\int_\Omega |\nabla\varphi|^2dV_g+\frac{n-2}{4(n-1)}\int_\Omega R\varphi^2dV_g+\frac{n-2}{2}\int_\Sigma H(\Tr\varphi)^2d\sigma_g,
\end{equation}
where $\Tr:W^{1,2}({M})\to W^{\frac12,2}(\partial M)$ is the trace map and $dV_g$ and $d\sigma_g$ are the volume forms induced by $g$ on $M$ and on $\partial M$ respectively.
By using the assumptions $sp>n$ and $s\geq1$, one can show that $E(\varphi)$ is finite for each $\varphi\in W^{1,2}(M)$.

Let $\bar{q}=\frac{n}{n-2}$.
Then for $2\leq q\leq 2\bar{q}$, and $2\leq r\leq \bar{q}+1$ with $q\geq r$, and for $b\in\R$, we define
\begin{equation}
\Yam^{q,r}_{b}(\Omega,\Sigma)=\inf_{\varphi\in B^{q,r}_{b}(\Omega,\Sigma)}E(\varphi),
\end{equation}
where
\begin{equation}
B^{q,r}_{b}(\Omega,\Sigma)=\{\varphi\in W^{1,2}(\Omega,\Sigma):\|\varphi\|_{L^q(\Omega)}^q+b\|\Tr\varphi\|_{L^r(\Sigma)}^r=1\},
\end{equation}
and
\begin{equation}
    W^{1,2}(\Omega,\Sigma)=\overline{\left\{\varphi\in C^1(\Omega)\cap C(\overline{\Omega}): \varphi\vert_{M\setminus\Omega}\equiv 0\ \text{and}\ (\Tr \varphi)|_{\partial M\setminus\Sigma}\equiv 0\right\}},
\end{equation}
with the closure taken in $W^{1,2}(M)$.

Note that $q=2\bar{q}$ is the critical exponent of the embedding $W^{1,2}(M)\hookrightarrow L^q(M)$,
while $r=\bar{q}+1$ is the critical exponent for the continuous trace operator $W^{1,2}(M)\hookrightarrow L^r(\partial M)$.

A primary aim of this section is to establish that the sign of $\Yam^{q,r}_{b}(\Omega,\Sigma)$ is a conformal invariant and does not depend on the indices.
To this end, we start by stating some lemmata.

\begin{lemma}[\cite{Esco96a}]
\label{l:polynomial}
Let $q>r>1$, $a>0$ and $b$ be constants (if $b>-a$ we can have $q=r$), and let
\begin{equation}
f_b(x) = ax^q+bx^r ,
\end{equation}
where $b\in\R$ is a parameter.
Then we have the following.
\begin{enumerate}[(a)]
\item
The equation $f_b(x)=1$ has a unique positive solution $x_b>0$.
\item
The correspondence $b\mapsto x_{b}$ is continuous.
\end{enumerate}
\end{lemma}

The next lemma is adapted from Proposition 2.3 of \cite{Esco96a}.

\begin{lemma}
\label{l:estimates_L2norms}
Let $2\leq q\leq 2\bar{q}$, and $2\leq r\leq \bar{q}+1$ with $q>r$, and let $b\in\R$.
Then for any $\eps>0$, there exists a constant $C_\eps\geq0$ such that
\begin{equation}
\label{FirstEscobarEstimate}
\|\varphi\|_{L^2(\Omega)}^2 \leq \eps \|\nabla\varphi\|_{L^2(\Omega)}^2 + C_\eps,
\end{equation}
and
\begin{equation}
\label{SecondEscobarEstimate}
\|\Tr\varphi\|_{L^2(\Sigma)}^2 \leq \eps \|\nabla\varphi\|_{L^2(\Omega)}^2 + C_\eps,
\end{equation}
for all $\varphi\in B^{q,r}_{b}(\Omega,\Sigma)$.
\begin{proof}
We will prove that the inequalities hold for all $\varphi$ satisfying
$$
\begin{cases}
\varphi\in C^1(\Omega)\cap C(\overline\Omega),\\
\supp\,\varphi\subset \Omega\cup\Sigma,\\
\|\varphi\|_{L^q(\Omega)}^q+b\|\Tr\varphi\|_{L^r(\Sigma)}^r=:F(\varphi)<2.
\end{cases}
$$
Let us first prove (\ref{FirstEscobarEstimate}). We start by using H\"older's inequality
\begin{equation*}
\begin{split}
    \int_{\Omega}\varphi^2 dV 
    &\leq \left[\int_{\Omega}|\varphi|^q dV \right]^{\frac{2}{q}}\left[\int_{\Omega}1 dV\right]^{\frac{q-2}{q}} \\
     &= \left[\int_{\Omega}|\varphi|^qdV \right]^{\frac{2}{q}}\vol(\Omega)^{\frac{q-2}{q}}\\
     &\leq \left[2-b\int_{\Sigma}|\gamma\varphi|^rd\sigma \right]^{\frac{2}{q}} \vol(\Omega)^{\frac{q-2}{q}}\\
     &\leq \left[2+|b|\int_{\Sigma}|\gamma\varphi|^rd\sigma \right]^{\frac{2}{q}}\vol(\Omega)^{\frac{q-2}{q}}.
\end{split}
\end{equation*}
Since $\frac{2}{q}\leq 1$, we have $(a+b)^{\frac{2}{q}}\leq a^{\frac{2}{q}}+b^{\frac{2}{q}}$ for $a,b\geq 0$, and so
\begin{equation*}
    \begin{split}
    \left[2+|b|\int_{\Sigma}|\gamma\varphi|^rd\sigma \right]^{\frac{2}{q}}\vol(\Omega)^{\frac{q-2}{q}} 
    \leq & \left[2^{\frac{2}{q}}+|b|^{\frac{2}{q}}\left[\int_{\Sigma}|\gamma\varphi|^rd\sigma\right]^{\frac{2}{q}}\right]
    \vol(\Omega)^{\frac{q-2}{q}}\\
    \leq & 2^{\frac{2}{q}}\vol(\Omega)^{\frac{q-2}{q}}+2^{\frac{q}{2}}\vol(\Omega)^{\frac{q-2}{q}}|b|^{\frac{2}{q}}
    \left[\int_{\Sigma}|\gamma\varphi|^rd\sigma \right]^{\frac{2}{q}},
    \end{split}
\end{equation*}
yielding
\begin{equation}
\label{eq:former_star}
    \left[2+|b|\int_{\Sigma}|\gamma\varphi|^rd\sigma \right]^{\frac{2}{q}}\vol(\Omega)^{\frac{q-2}{q}} \leq  C_0+C_0|b|^{\frac{q}{2}}\left[\int_{\Sigma}|\gamma\varphi|^rd\sigma \right]^{\frac{2}{q}},
\end{equation}
where $C_0=2^{\frac{2}{q}}\vol(\Omega)^{\frac{q-2}{q}}$.

Now, the trace inequality says that there is a positive constant $C_1=C_1(M,g,r)$ such that
\begin{equation*}
\begin{split}
    \left[\int_{\Sigma}|\gamma \varphi|^rdV\right]^{\frac{2}{r}} =  \left[\int_{\partial M}|\gamma \varphi|^rdV\right]^{\frac{2}{r}} 
    &\leq C_1\int_M \big( |\nabla \varphi|^2+\varphi^2\big) dV 
    = C_1\int_\Omega \big( |\nabla \varphi|^2+\varphi^2\big) dV,
\end{split}
\end{equation*}
which leads to
\begin{equation}
\label{eq:former_starstar}
    \begin{split}
    \left[\int_{\Sigma}|\gamma \varphi|^rdV\right]^{\frac{2}{q}} 
    &\leq C_1^{\frac{r}{q}}\left[\int_\Omega |\nabla \varphi|^2dV+\int_\Omega\varphi^2 dV\right]^{\frac{r}{q}} \\
     &\leq  C_1^{\frac{r}{q}}\left(\int_\Omega |\nabla \varphi|^2dV\right)^{\frac{r}{q}}+C_1^{\frac{r}{q}}\left(\int_\Omega\varphi^2 dV\right)^{\frac{r}{q}},
\end{split}
\end{equation}
since $\frac{r}{q}\leq 1$.
Plugging (\ref{eq:former_starstar}) into (\ref{eq:former_star}) we get
\begin{equation}
\begin{split}
    \int_\Omega \varphi^2 dV 
    & \leq  C_0+C_0|b|^{\frac{2}{q}}\left[\int_{\Sigma}|\gamma\varphi|^rd\sigma\right]^{\frac{2}{q}} \\
     & \leq C_0+C_0|b|^{\frac{2}{q}}C_1^{\frac{r}{q}}\left[\left(\int_\Omega |\nabla\varphi|^2 dV\right)^{\frac{r}{q}}+\left(\int_\Omega \varphi^2dV\right)^{\frac{r}{q}}\right]\\
     &=C_0+C_3\left[\left(\int_\Omega |\nabla\varphi|^2 dV\right)^{\frac{r}{q}}+\left(\int_\Omega \varphi^2dV\right)^{\frac{r}{q}}\right],
\end{split}
\end{equation}
where we introduced $C_3=C_0|b|^{\frac{2}{q}}C_1^{\frac{r}{q}}>0$.

Given $\epsilon_1>0$, there is $C_4(\epsilon_1)$ such that $t^{\frac{r}{q}}\leq \left(\eps_1/C_3\right)t+C_4$ for $t>0$. 
Applying this with $t=\int_\Omega \varphi^2dV$, we get
$$
\begin{array}{ccl}
    \int_{\Omega}\varphi^2dV &\leq & C_0+C_3\left(\int_\Omega |\nabla\varphi|^2dV \right)^{\frac{r}{q}}+\frac{\epsilon_1}{C_3} C_3 \int_\Omega\varphi^2dV+C_4.
\end{array}
$$
Then with $t=\int_\Omega|\nabla \varphi|^2dV$, we have
$$
\begin{array}{cccl}
    &\int_{\Omega}\varphi^2dV &\leq & C_0+\epsilon_1 \int_\Omega |\nabla\varphi|^2dV+C_4 +\epsilon_1\int_\Omega\varphi^2dV+C_4\\
    \Rightarrow&(1-\epsilon_1)\int_\Omega\varphi^2dV&\leq& \epsilon_1\int_\Omega |\nabla\varphi|^2dV+C_0+2C_4\\
    \Rightarrow&\int_\Omega \varphi^2dV &\leq& \frac{\epsilon_1}{1-\epsilon_1}\int_\Omega|\nabla\varphi|^2dV+C_5 
\end{array}
$$
where $C_5:=\frac{C_0+2C_4}{1-\epsilon_1}$. 
This gives (\ref{FirstEscobarEstimate}).

Now, again by the Sobolev trace inequality, there is $C_7>0$ such that
$$
\int_\Sigma (\Tr\varphi)^2d\sigma=\int_{\partial M}(\Tr\varphi)^2d\sigma\leq C_7\left(\int_M\varphi^2dV+\int_M\left|\nabla\left(\varphi^2\right)\right|dV\right)=C_7\left(\int_\Omega\varphi^2dV+\int_\Omega\left|\nabla\left(\varphi^2\right)\right|dV\right).
$$
However, if $\epsilon_2>0$,
\begin{equation}
\label{eq:gradient_u_2}
|\nabla(u^2)|=|2u\nabla u|\leq 2|u||\nabla u|\leq \epsilon_2|\nabla u|^2+\frac{u^2}{\epsilon_2},
\end{equation}
because $\left(\sqrt{\epsilon_2} |\nabla u|-\frac{u}{\sqrt {\epsilon_2}}\right)^2\geq 0$.

But Inequality (\ref{FirstEscobarEstimate}) says that there is $C_6$ such that
\begin{equation}
\label{eq:former_starstarstar}
\int_\Omega \varphi^2dV\leq \epsilon^2_2\int_\Omega|\nabla \varphi|^2dV+C_6,
\end{equation}
which combined with (\ref{eq:gradient_u_2}) and (\ref{eq:former_starstarstar}) implies that
$$
\begin{array}{ccl}
    \int_\Sigma (\Tr\varphi)^2 d\sigma & \leq& C_7\epsilon^2_2\int_{\Omega}|\nabla\varphi|^2dV+C_7C_6+C_7\epsilon_2\int_{\Omega}|\nabla\varphi|^2dV+\frac{C_7}{\epsilon_2}\int_{\Omega}\varphi^2dV \\
     &\leq & C_7(\epsilon^2_2+2\epsilon_2)\int_\Omega |\nabla\varphi|^2dV+C_7C_6+\frac{C_7C_6}{\epsilon_2},
\end{array}
$$
again by (\ref{eq:former_starstarstar}). The estimate (\ref{SecondEscobarEstimate}) follows from taking $\epsilon=C_7(\epsilon_2^2+2\epsilon_2)$ and $C_\epsilon=\max\{C_7C_6+\frac{C_7C_6}{\epsilon_2}, C_5(\epsilon)\}$.
\end{proof}
\end{lemma}

Later we are going to use the fact that if $C_3$ decreases, $C_4$ may decrease as well and thus $C_\epsilon$ may decrease. That is, if we consider $C_\epsilon=C_\epsilon(b)$, we can assume that $|b_1|>|b_2|\Rightarrow C_\epsilon(b_1)\geq C_\epsilon(b_2)$.

The following result will be important.

\begin{lemma}\label{t:generalities_y}
Let $2\leq q\leq 2\bar{q}$, and $2\leq r\leq \bar{q}+1$ with $q\geq r$, and let $b\in\R$.
Then we have the following.
\begin{enumerate}[(a)]
\item Given $\epsilon>0$, there exists $K_\epsilon>0$ such that
\begin{equation}
\label{eq:R_estimate_L2}
    \left|\int_\Omega R\varphi^2dV_g\right|\leq \epsilon ||\varphi||^2_{W^{1,2}(\Omega)}+K_\epsilon||\varphi||^2_{L^2(\Omega)},
\end{equation}
and that
\begin{equation}
\label{eq:H_estimate_L2}
    \left|\int_\Sigma H(\gamma\varphi)^2d\sigma_g\right|\leq \epsilon ||\gamma\varphi||^2_{W^{\frac{1}{2},2}(\Sigma)}+ K_\epsilon||\gamma\varphi||^2_{L^2(\Sigma)}.
\end{equation}
\item For $q>r$, if $B^{q,r}_b(\Omega,\Sigma)\neq \emptyset$, 
$\Yam^{q,r}_{b}(\Omega,\Sigma)$ is finite.
\item
Again, if $q>r$, there are constants $C$ and $K$ such that 
\begin{equation}
\|\nabla\varphi\|_{L^2(\Omega)}^2 \leq K E(\varphi) + C ,
\end{equation}
for all $\varphi\in B^{q,r}_{b}(\Omega,\Sigma)$.
\item
The quantity $\Yam_g(\Omega,\Sigma):=\Yam^{2\bar{q},r}_0(\Omega,\Sigma)$ is a conformal invariant, that is, $\Yam_g(\Omega,\Sigma)=\Yam_{\tilde{g}}(\Omega,\Sigma)$ for any two metrics $\tilde{g}\sim g$ of class $W^{s,p}$.
We refer to $\Yam_g(\Omega,\Sigma)$ as {\em the} relative Yamabe invariant of $(\Omega,\Sigma)$.
\item
The quantity $\Yam^{2\bar{q},\bar{q}+1}_b(\Omega,\Sigma)$ is also a conformal invariant.
\end{enumerate}
\end{lemma}

\begin{proof}
To prove (a) we first refer to the discussion on page 9 of \cite{HT13} to justify that we can bound the curvature terms of $E(\varphi)$:

    If $\varphi\in W^{(1-\delta),2}(M)$ for a sufficiently small $\delta>0$, we have $\varphi^2\in W^{2-s,p'}(M)=\left[W^{s-2,p}(M)\right]'$ for $sp>n$ and $s\geq 1$, with $p'$ the H\"older conjugate of $p$. In particular, if $\varphi\in W^{1,2}(\Omega,\Sigma)$, there is $K_1$ such that
    \begin{equation}
    \left|\int_\Omega R\varphi^2dV_g\right|=\left|\int_M R\varphi^2dV_g\right|\leq K_1||R||_{W^{s-2,p}(M)}||\varphi^2||_{W^{2-s,p'}(M)}.    
    \end{equation}
    However the pointwise multiplication is bounded as a map $W^{(1-\delta),2}(M)\otimes W^{(1-\delta),2}(M)\to W^{2-s,p'}(M)$, so there is $K_2$ a constant providing
\begin{equation}
    ||\varphi^2||_{W^{2-s,p'}(M)}\leq K_2||\varphi||_{W^{(1-\delta),2}(M)}^2,
\end{equation}
    for a suitably small $\delta$.
    
    It follows, by an interpolation argument, that one gets for all $ \epsilon>0$ a positive $K_\epsilon$ satisfying
    \begin{equation}
    \label{eq:the_item_before_interpolation}
    \left|\int_\Omega R\varphi^2dV_g\right|\leq \epsilon ||\varphi||^2_{W^{1,2}(\Omega)}+K_\epsilon||\varphi||^2_{L^2(\Omega)},    
    \end{equation}
    proving the first estimate in (a).
    
    If $q>r$, we can combine this with (\ref{FirstEscobarEstimate}) and get that for all $\epsilon>0$, there is $L_\epsilon>0$ such that
    \begin{equation}
    \left|\int_\Omega R\varphi^2dV_g\right|\leq \epsilon ||\nabla\varphi||^2_{L^2(\Omega)}+L_\epsilon.    
    \end{equation}
    Since $\varphi\in W^{1,2}(\Omega)$ implies $(\gamma \varphi)\in W^{\frac{1}{2},2}(\partial M)$, Corollary A.5 of \cite{HT13} says $(\gamma \varphi)^2\in W^{\frac{1}{p}+1-s, p'}(\partial M)$ if
    $$
    \begin{cases}
    \sigma = \frac{1}{p}+1-s\leq \frac{1}{2}\\
    \sigma-\frac{n-1}{p'}=\frac{1}{p}+1-s-\frac{n}{p'}+\frac{1}{p'}=2-s-\frac{n}{p'}\leq \frac{1}{2}-\frac{n-1}{2}
    \end{cases}
    $$
    which can be rewritten as
    $$
    \begin{cases}
    ps\geq \frac{p}{2}+1\\
    \frac{ps}{n}\geq 1+\frac{p}{n}-\frac{p}{2}.
    \end{cases}
    $$
    Now, since $sp>n>2$, the second condition is satisfied since $\frac{ps}{n}>1>1+\frac{p}{n}-\frac{p}{2}$. Adding that $s\geq 1$ we get that the first condition is satisfied because
    $$
    \frac{p}{2}+1<\frac{ps}{2}+1\leq \frac{ps}{2}+\frac{ps}{n}<ps.
    $$
    So, as in the derivation of equation (\ref{eq:the_item_before_interpolation}) above, there is $\tilde K_\epsilon>0$ such that
\begin{equation}
    \left|\int_\Sigma H(\gamma\varphi)^2d\sigma_g\right|\leq \epsilon ||\gamma\varphi||^2_{W^{\frac{1}{2},2}(\Sigma)}+\tilde K_\epsilon||\varphi||^2_{L^2(\Sigma)},
\end{equation}
which is the second estimate in (a).
    
Using the trace inequality followed by (\ref{SecondEscobarEstimate}), if $q>r$ we get that given $\epsilon>0$, there is $L_\epsilon>0$ providing
\begin{equation}
\left|\int_\Sigma H(\gamma\varphi)^2d\sigma_g\right|\leq \epsilon ||\nabla\varphi||^2_{W^{1,2}(\Omega)}+ L_\epsilon.
\end{equation}

We can finally prove (c), which is a lower regularity version of proposition 2.4 in \cite{Esco96a}:
$$
\begin{array}{ccl}
    ||\nabla\varphi||^2_{L^2(\Omega)} &=& E(\varphi)-\frac{n-2}{4(n-1)}\int_\Omega R\varphi^2dV_g-\frac{n-2}{2}\int_\Sigma H (\gamma\varphi)^2d\sigma_g  \\
     &\leq& E(\varphi)+2\epsilon||\nabla \varphi||^2_{L^2(\Omega)}+2L_\epsilon,
\end{array}
$$
for a suitable choice of $\epsilon$ and $L_\epsilon$. Thus, making sure $\epsilon<\frac{1}{2}$, we get
$$
||\nabla\varphi||^2_{L^2(\Omega)}\leq \frac{1}{1-2\epsilon}E(\varphi)+\frac{2}{1-2\epsilon}L_\epsilon
$$
and the result follows by choosing $K=\frac{1}{1-2\epsilon}$ and $C=\frac{2}{1-2\epsilon}L_\epsilon$.

With this in hand, it is simple to prove (b), since for all $\varphi$:
$$
E(\varphi)\geq \frac{1}{K}||\nabla\varphi||^2_{L^2(\Omega)}-\frac{C}{K}\geq -\frac{C}{K}.
$$

For (d) and (e), let $\tilde g=\phi^{\frac{4}{n-2}}g$ and let us see how each element of the expression of $E(\varphi)$ transforms between the two metrics:
$$
\begin{cases}
    dV_{\tilde g}=\phi^{\frac{2n}{n-2}}dV_g=\phi^{2\bar q}dV_g\\
    d\sigma_{\tilde g}=\phi^{\frac{2n-2}{n-2}}d\sigma_g=\phi^{\bar q+1}d\sigma_g\\
    R_{\tilde g}=\phi^{-\frac{n+2}{n-2}}\left(-\frac{4(n-1)}{n-2}\Delta_g \phi+R_g\phi\right)\\
    H_{\tilde g}=\phi^{-\frac{n}{n-2}}\left(H_g\phi+\frac{2}{n-2}\frac{\partial \phi}{\partial \eta} \right)\\
    |\nabla \varphi|^2_{\tilde g}=\phi^{-\frac{4}{n-2}}|\nabla\varphi|_g^2
\end{cases}
$$
The next step is to prove that $E_{\tilde g}(\varphi)=E_g(\phi\varphi)$:
$$
\begin{array}{ccl}
    E_{\tilde g}(\varphi) &=& \int_{\Omega}\left(|\nabla\varphi|^2_{\tilde g}+\frac{n-2}{4(n-1)}R_{\tilde g}\varphi^2\right)dV_{\tilde g}+\frac{n-2}{2}\int_{\Sigma}H_{\tilde g}(\gamma\phi)^2 d\sigma_{\tilde g}  \\
    & = & \int_{\Omega}\left(\phi^{-\frac{4}{n-2}}|\nabla\varphi|^2_{g}+\frac{n-2}{4(n-1)}\phi^{-\frac{n+2}{n-2}}\left(-\frac{4(n-1)}{n-2}\Delta_g \phi+R_g\phi\right)\varphi^2\right)\phi^{\frac{2n}{n-2}}dV_{ g}\\
    &&+\frac{n-2}{2}\int_{\Sigma}\phi^{-\frac{n}{n-2}}\left(H_g\phi+\frac{2}{n-2}\frac{\partial \phi}{\partial \eta} \right)(\gamma\varphi)^2 \phi^{\frac{2n-2}{n-2}}d\sigma_{\tilde g}\\
    &=&\int_{\Omega}\left(\phi^2|\nabla\varphi|^2_{g}-\phi\varphi^2\Delta_g \phi+\frac{n-2}{4(n-1)}R_g(\phi\varphi)^2\right)dV_{ g}+\int_{\Sigma}\frac{n-2}{2} H_g\left(\phi(\gamma\varphi)\right)^2+\phi(\gamma\varphi)^2\frac{\partial \phi}{\partial \eta}  d\sigma_{\tilde g}\\
    &=&E_g(\phi\varphi)+\int_{\Omega}\left(-|\nabla(\phi\varphi)|_g^2+\phi^2|\nabla\varphi|^2_g-\phi\varphi^2\Delta_g\phi\right)dV_g+\int_\Sigma\phi(\gamma\varphi)^2\frac{\partial \phi}{\partial \eta}  d\sigma_{\tilde g}\\
    &=&E_g(\phi\varphi)+\int_{\Omega}\left(-|\nabla(\phi\varphi)|_g^2+\phi^2|\nabla\varphi|^2_g+\nabla(\phi\varphi^2)\cdot\nabla\phi\right)dV_g
\end{array}
$$
where we used Green's first identity in the last equality. Then if we use the identities
\begin{equation}
|\nabla(\phi\varphi)|_g^2=\phi^2|\nabla\varphi|^2_g+\varphi^2|\nabla\phi|^2_g+2\phi\varphi\nabla\phi\cdot\nabla\varphi,
\end{equation}
and
\begin{equation}
\left(\nabla(\phi\varphi^2)\right)\cdot\nabla\phi=\varphi^2|\nabla\phi|^2_g+2\phi\varphi\nabla\varphi\cdot\nabla\phi, 
\end{equation}
we get that
\begin{equation}
    E_{\tilde{g}}(\varphi)=E_g(\phi\varphi).
\end{equation}

Finally, the transformation of the respective measures under the conformal change of the metric implies that
$$
\int_\Omega |\varphi|^{2\bar q} dV_{\tilde g}=\int_{\Omega}|\phi\varphi|^{2\bar q}dV_g
$$
and
$$
\int_\Sigma |\gamma\varphi|^{\bar q+1} d\sigma_{\tilde g}=\int_{\Sigma}|\phi(\gamma\varphi)|^{\bar q+1}d\sigma_g
$$
so $\varphi\in B^{2\bar q,r}_{0,\tilde g}(\Omega,\Sigma)$ if, and only if, $\phi\varphi\in B^{2\bar q,r}_{0,g}(\Omega,\Sigma)$, proving (c), and $\varphi\in B^{2\bar q,\bar q+1}_{b,\tilde g}(\Omega,\Sigma)$ if, and only if, $\phi\varphi\in B^{2\bar q,\bar q+1}_{b,g}(\Omega,\Sigma)$, proving (d).
\end{proof}

\begin{theorem}\label{t:subcrit}
Let ${M}$ be a smooth connected Riemannian manifold with dimension $n\geq 3$ and with a metric $g\in W^{s,p}$, 
where we assume $sp>n$ and $s\geq 1$.
Let $q\in[2,2\bar{q})$, and $r\in[2,\bar{q}+1)$ with $q>r$. Let also $b\in\R$ and let $\Omega$ and $\Sigma$ be relatively open sets.
Then, there exists a strictly positive function $\phi\in B^{q,r}_b(\Omega, \Sigma)\cap W^{s,p}({M})$, such that
\begin{equation}\label{e:yamabe}
\begin{cases}
-\Delta\phi+\frac{n-2}{4(n-1)}R\phi&=\lambda q\phi^{q-1},\ in\ \Omega\\
\Tr\partial_\nu\phi+\frac{n-2}2H\Tr\phi&=\lambda rb(\Tr\phi)^{r-1},\ in\ \Sigma
\end{cases}
\end{equation}
where the sign of $\lambda$ is the same as that of $\Yam_b^{q,r}(\Omega, \Sigma)$ defined above.
\end{theorem}

\begin{proof}
The above equation is the Euler-Lagrange equation for the functional $E$ over positive functions with the Lagrange multiplier $\lambda$, so it suffices to show that $E$ attains its infimum $\Yam_b^{q,r}(\Omega,\Sigma)$ over $B_b^{q,r}(\Omega,\Sigma)$ at a positive function $\phi\in W^{s,p}({M})$ (notice that the existence of a minimizer does not depend on the sets $\Omega$ and $\Sigma$ being open).

Let $\{\phi_i\}_i\subset B_b^{q,r}(\Omega,\Sigma)$ be a sequence satisfying $E(\phi_i)\to\Yam_b^{q,r}(\Omega,\Sigma)$. Since $E(\varphi)=E(|\varphi|)$, for all $\varphi$, we can assume $\phi_i\geq 0,\forall i$.
If $\varphi\in B_b^{q,r}(\Omega,\Sigma)$ satisfies the bound $E(\varphi)\leq \Lambda$, then inequality (\ref{FirstEscobarEstimate}) and item (c) of Lemma \ref{t:generalities_y} tell us there is $C(\Lambda)>0$ such that $\|\varphi\|_{W^{1,2}(\Omega)}\leq C(\Lambda)$.
Now, since $\Yam_b^{q,r}(\Omega,\Sigma)$ is finite, we conclude that $\{\phi_i\}_i$ is bounded in $W^{1,2}({\Omega})$.
By the reflexivity of $W^{1,2}({\Omega})$, the compactness of $W^{1,2}({\Omega})\hookrightarrow L^{q}({\Omega})$,
and the compactness of the trace map $\Tr:W^{1,2}({\Omega})\hookrightarrow L^{r}(\Sigma)$,
there exists an element $\phi\in W^{1,2}({\Omega})$ and a subsequence $\{\phi'_i\}\subset\{\phi_i\}$ such that, for a fixed small $\delta>0$:
\begin{enumerate}[(i)]
    \item $\phi'_i\rightharpoonup\phi$ in $W^{1,2}({\Omega})$,
    \item $\phi'_i\to\phi$ in $W^{1-\delta,2}(\Omega)$, 
    \item $\phi'_i\rightarrow\phi$ in $L^{q}({\Omega})$,
    \item $\phi'_i\rightarrow\phi$ in $L^{2}({\Omega})$,
    \item $\Tr\phi'_i\to \Tr\phi$ in $W^{\frac{1}{2}-\delta, 2}$,
    \item $\Tr\phi'_i\rightarrow\Tr\phi$ in $L^{r}(\Sigma)$.
\end{enumerate}
Items (iii) and (vi) imply $\phi\in B_b^{q,r}(\Omega,\Sigma)$, hence $\phi\geq 0$.

We have to prove that $E(\phi)\leq \lim E(\phi'_i)=\mathcal{Y}_b^{q,r}(\Omega,\Sigma)$, because then $E(\phi)=\mathcal{Y}^{q,r}_b(\Omega,\Sigma)$ and thus $\phi$ satisfies \eqref{e:yamabe}.

Indeed, the proof of Lemma \ref{t:generalities_y} (a) shows that $\varphi\mapsto \langle R, \varphi\rangle_{\Omega}$ is continuous in $W^{1-\delta,2}(\Omega)$ and $\varphi\mapsto  \langle H,\varphi \rangle_{\Sigma}$ is continuous in $W^{\frac{1}{2}-\delta, 2}(\Sigma)$. So the maps $\varphi\mapsto \langle R,\varphi^2\rangle_{\Omega}$ and $\varphi\mapsto \langle H,\varphi^2\rangle_\Sigma$ have, respectively, the same regularity. Therefore, by (ii) and (v), we have
\begin{equation}
    E(\phi)-\mathcal{Y}^{q,r}_b(\Omega,\Sigma)=E(\phi)-\lim E(\phi'_i)=\|\nabla \phi\|^2_{L^2(\Omega)}-\lim \|\nabla\phi'_i\|^2_{L^2(\Omega)}
\end{equation}
and so by (i) and (iv), $E(\phi)\leq \mathcal{Y}^{q,r}_b(\Omega,\Sigma)$ as expected.

Corollary B.4 of \cite{HT13} implies that $\phi\in W^{s,p}({M})$, as is detailed in the proof of Lemma \ref{l:regularity} below. 
Also, since $\phi\in B^{q,r}_b(\Omega,\Sigma)$, we have $\phi\not\equiv 0$. But, if we call $N=(\gamma\phi)^{-1}(\{0\})$, Lemma B.7 of \cite{HT13} tells us that $\phi>0$.

Finally, multiplying \eqref{e:yamabe} by $\phi$ and integrating by parts, we conclude that the sign of the Lagrange multiplier $\lambda$ is the same as that of $\Yam_b^{q,r}(\Omega, \Sigma)$.
\end{proof}

Under the conformal scaling $\tilde{g}=\varphi^{2\bar{q}-2}g$, the scalar curvature and the mean extrinsic curvature transform as
\begin{equation}
\label{eq:constant_sign_curvatures}
\begin{cases}
\tilde{R}&\textstyle=\varphi^{1-2\bar{q}}(-\frac{4(n-1)}{n-2}\Delta\varphi+R\varphi),\\
\tilde{H}&\textstyle=(\Tr\varphi)^{-\bar{q}}(\frac{2}{n-2}\Tr\partial_\nu\varphi+H\Tr\varphi),
\end{cases}
\end{equation}
so assuming the conditions of the above theorem we infer that any given metric $g\in W^{s,p}$ 
can be transformed to the metric $\tilde{g}=\phi^{2\bar{q}-2}g$ with the continuous scalar curvature $\tilde{R}=\frac{4\lambda q(n-1)}{n-2}\phi^{q-2\bar{q}}$ in $\Omega$,
and the continuous mean curvature $\tilde{H}=\frac{2\lambda br}{n-2}(\Tr\phi)^{r-\bar{q}-1}$ on $\Sigma$,
where the conformal factor $\phi$ is as in the theorem.
In other words, given any metric $g\in W^{s,p}$,
there exist $\phi \in W^{s,p}({M})$ with $\phi>0$, $\tilde{R}\in W^{s,p}({M})$ and $\tilde{H}\in W^{s-\frac1p,p}(\Sigma)$, having {\em constant sign}, such that
\begin{equation}\label{e:conformal-transformation}
\begin{cases}\textstyle
-\frac{4(n-1)}{n-2}\Delta\phi+R\phi&=\tilde{R}\phi^{2\bar{q}-1},\\
\textstyle
\frac2{n-2}\Tr\partial_\nu\phi+H\Tr\phi&=\tilde{H}(\Tr\phi)^{\bar{q}}.
\end{cases}
\end{equation}
We will prove below that the conformal invariant $\Yam_g$ of the metric $g$ completely determines the sign of $\tilde{R}$, independent of the indices $q$, $r$ or $b$,
giving rise to a Yamabe classification of metrics in $W^{s,p}$.
Note that the sign of the boundary mean curvature can be controlled by the sign of the parameter $b\in\R$,
unless of course $\tilde{R}\equiv0$, in which case we are forced to have $\tilde{H}\equiv0$ in the above argument.

We will need a lemma.
\begin{lemma}\label{l:b_mononicity}
If $q\in [2,2\bar q]$, $r\in[2,\bar q+1]$, $q>r$, then the map $b\mapsto \mathcal{Y}^{q,r}_b(\Omega,\Sigma)$ is non-increasing.
\begin{proof}
Let $b'<b$ and $\varphi\in B_b^{q,r}(\Omega, \Sigma)$. Then,as a consequence of Lemma \ref{l:polynomial}, there exists $k_\varphi>0$ such that $k_\varphi\varphi\in B_{b'}^{q,r}(\Omega, \Sigma)$. That means
$$
\begin{array}{ccl}
    1&=&k_\varphi^q||\varphi||^q_{L^q(\Omega)}+k_\varphi^rb'||\gamma\varphi||^r_{L^r(\Sigma)} \\
     & =&||\varphi||^q_{L^q(\Omega)}+b||\gamma\varphi||^r_{L^r(\Sigma)}\\
     &=&||\varphi||^q_{L^q(\Omega)}+b'||\gamma\varphi||^r_{\Sigma, r}+(b-b')||\gamma\varphi||^r_{L^r(\Sigma)},
\end{array}
$$
and so, since $b-b'> 0$, we get
$$
k_\varphi^q||\varphi||^q_{L^q(\Omega)}+k_\varphi^rb'||\gamma\varphi||^r_{L^r(\Sigma)}\geq ||\varphi||^q_{L^q(\Omega)}+b'||\gamma\varphi||^r_{L^r(\Sigma)},
$$
implying $k_\varphi\geq 1$.

Now, we have $E(k_\varphi\varphi)=k_\varphi^2E(\varphi)\geq E(\varphi)$. Hence $\mathcal{Y}^{q,r}_{b'}(\Omega, \Sigma)\geq \mathcal{Y}^{q,r}_b(\Omega, \Sigma)$, and thus the map is non-increasing.
\end{proof}
\end{lemma}

\begin{theorem}\label{t:yclass}
Let $({M},g)$ be a smooth, compact, connected, $n$-dimensional Riemannian manifold with boundary,
where we assume that the components of the metric $g$ are (locally) in $W^{s,p}$, with $sp>n$, $s\geq 1$, and $n\geq 3$.
Let $\Omega\subset M$ and $\Sigma\subset\partial M$ be (relatively) closed sets.
Then, the sign of $\Yam^{q,r}_{b}(\Omega,\Sigma)$ is independent of $q\in[2,2\bar{q}]$, $r\in[2,\bar{q}+1]$, and $b\in\R$, as long as $q>r$.
In particular, this sign is a conformal invariant.
\end{theorem}

\begin{proof}
Let $q,q'\in[2,2\bar{q}]$, $r,r'\in[2,\bar{q}+1]$, and $b,b'\in\R$, with $q>r$ and $q'>r'$.

Suppose $\Yam^{q,r}_{b}(\Omega,\Sigma)<0$.
Then $E(\varphi)<0$ for some $\varphi\in B^{q,r}_{b}(\Omega,\Sigma)$.
By scaling, there is some $k>0$ satisfying $k\varphi\in B^{q',r'}_{b'}(\Omega,\Sigma)$.
Now since $E(k\varphi)=k^2E(\varphi)$, we infer that $\Yam^{q',r'}_{b'}(\Omega,\Sigma)<0$.

Suppose that $\Yam^{q,r}_{b}(\Omega,\Sigma)\geq0$.
Then $E(\varphi)\geq0$ for all $\varphi\in B^{q,r}_{b}(\Omega,\Sigma)$.
On the other hand, for any $\psi\in B^{q',r'}_{b'}(\Omega,\Sigma)$, there is some $k>0$ such that $k\psi\in B^{q,r}_{b}(\Omega,\Sigma)$,
and hence $E(\psi)=k^{-2}E(k\psi)\geq0$.
This implies that $\Yam^{q',r'}_{b'}(\Omega,\Sigma)\geq0$.

What remains is to prove that $\Yam^{q,r}_{b}(\Omega,\Sigma)>0$ implies $\Yam^{q',r'}_{b'}(\Omega,\Sigma)>0$.
Since the parameters are arbitrary, in light of the preceding paragraph, this would also establish the equivalence between 
$\Yam^{q,r}_{b}(\Omega,\Sigma)=0$ and $\Yam^{q',r'}_{b'}(\Omega,\Sigma)=0$.

Let us start by fixing $q$ and $r$. Assume that $\Yam^{q,r}_{b}(\Omega,\Sigma)=\alpha>0$ and, by monotonicity, we can say $b<0$.
Let $\psi\in B^{q,r}_{b'}(\Omega,\Sigma)$.
Then there exists some $k>0$ such that $k\psi\in B^{q,r}_{b}(\Omega,\Sigma)$, and
\begin{equation}
\label{eq:E>0}
    E(k\psi)=k^{2}E(\psi)\geq \alpha.
\end{equation}

The value of $k$ comes from the equation
\begin{equation}
k^q\|\psi\|_{L^q(\Omega)}^q+bk^r\|\Tr\psi\|_{L^r(\Sigma)}^r = 1 .
\end{equation}
Let us focus first at $b'=0$. Then $||\varphi||^q_q=1$ for all $\varphi$ and
\begin{equation}
\label{e:k_from_polynomial}
k^{q-r} = k^{-r} -b\|\Tr\psi\|_{L^r(\Sigma)}^r
\end{equation}
If $\mathcal{Y}_0^{q,r}(\Omega,\Sigma)=0$, there is a minimizing sequence $\{\psi_m\}_m\subset B_0^{q,r}(\Omega, \Sigma)$ such that $E(\psi_m)\to 0$. But, if $k_m\psi_m\in B_b^{q,r}(\Omega,\Sigma)$, constraint (\ref{eq:E>0}) implies $k_m\to\infty$ and thus, by equation (\ref{e:k_from_polynomial}), $\|\Tr\psi_m\|^r_{L^r(\Sigma)}\to \infty$ for $b<0$.

On the other hand, from $E(\psi_m)\to 0$ it follows that there is $M>0$ such that $E(\psi_m)<M$ for all $m$ and Lemma \ref{t:generalities_y} (c) implies there are $K$ and $C$ such that $\|\nabla\psi_m\|^2_{L^2(\Omega)}\leq KM+C$. But by (\ref{FirstEscobarEstimate}), given $\epsilon>0$, there exists $C_\epsilon>0$ such that
\begin{equation}
    \|\psi_m\|^2_{L^2(\Omega)}\leq \epsilon \|\nabla\psi_m\|^2_{L^2(\Omega)}+C_\epsilon
\end{equation}
and hence the sequence $\left\{\|\psi_m\|^2_{W^{1,2}(\Omega)}\right\}_m$ is bounded and the continuity of the trace map implies that $\left\{\|\Tr \psi_m\|^2_{W^{\frac{1}{2}, 2}(\Sigma)}\right\}_m$ is bounded as well. But since $r<\bar q+1$ the Sobolev embedding theorem implies $W^{\frac{1}{2}, 2}(\Sigma)\hookrightarrow L^r(\Sigma)$ and there is a uniform bound on $\|\Tr \psi_m\|^r_{L^r(\Sigma)}$. Contradiction. So $\mathcal{Y}_0^{q,r}(\Omega,\Sigma)=\epsilon$ for some $\epsilon>0$.

Let us look now at what happens for $b'>0$. Assume $\mathcal{Y}^{q,r}_{b'}(\Omega,\Sigma)=0$ and let $\{\psi_m\}_m\subset B_{b'}^{q,r}(\Omega,\Sigma)$ be a minimizing sequence. If you take $b_2>0$, $k_m>0$ such that $k_m\psi_m\in B_{b_2}^{q,r}(\Omega,\Sigma)$, then $E(k_m\psi_m)=k_m^2E(\psi_m)$. So either $k_m\to\infty$ or $E(k_m\psi_m)\to 0$ and $\mathcal{Y}_{b_2}^{q,r}(\Omega,\Sigma)=0$. But $k_m$ satisfies
\begin{equation}
    k_m^q\|\psi_m\|^q_{L^q(\Omega)}+k_m^rb_2\|\Tr\psi_m\|^r_{L^r(\Sigma)}=1
\end{equation}
as a consequence, since $b_2>0$, $k_m\to \infty$ implies $\|\psi_m\|^q_{L^q(\Omega)}\to 0$ and $\|\Tr\psi_m\|^r_{L^r(\Sigma)}\to 0$. But that would contradict the fact that $\psi_m\in B_{b'}^{q,r}(\Omega,\Sigma)$, for all $m$. Hence $k_m$ is bounded and $\mathcal{Y}_{b_2}^{q,r}(\Omega,\Sigma)=0$, for all $b_2>0$.

Going back to the assumption $\mathcal{Y}_{b'}^{q,r}(\Omega,\Sigma)=0$ for some $b'>0$, let $\{b_m\}_m$ be a decreasing sequence of strictly positive numbers such that $b_m\to 0$. Then $\mathcal{Y}_{b_m}^{q,r}(\Omega,\Sigma)=0,\forall m$, and for each $m$, there is $\psi_m\in B_{b_m}^{q,r}(\Omega,\Sigma)$ such that $E(\psi_m)<\frac{1}{m}$. In particular, $E(\psi_m)\leq 1,\forall m$. Now, since the sequence of the $b_m$'s is bounded, the remark after the proof of lemma \ref{l:estimates_L2norms} allows us to use theorem \ref{t:generalities_y}, item (c) to get constants $K,C>0$ such that
\begin{equation}
    \|\nabla \psi_m\|^2_{L^2(\Omega)}\leq K+M
\end{equation}
which, just as above, implies $\left\{\|\Tr \psi_m\|^r_{L^r(\Sigma)}\right\}_m$ is uniformly bounded.

Finally, for each $m$ let $l_m>0$ be such that $l_m\psi_m\in B^{q,r}_0(\Omega, \Sigma)$. Then $E(l_m\psi_m)=l_m^2E(\psi_m)$ and
\begin{equation}
    l_m^q\|\psi_m\|^q_{L^q(\Omega)}=1
\end{equation}
so $l_m\to \infty$ iff $\|\psi_m\|_{L^q(\Omega)}\to 0$. However
\begin{equation}
    \|\psi_m\|^q_{L^q(\Omega)}+b_m\|\Tr\psi_m\|^r_{L^r(\Sigma)}=1.
\end{equation}
It follows, since $b_m\to 0$, that $\|\psi_m\|^q_{L^q(\Omega)}\to 0$ implies $\|\Tr\psi_m\|^r_{L^r(\Sigma)}\to \infty$, which cannot be. As a consequence, $\{l_m\}_m$ is uniformly bounded and $E(l_m\psi_m)\to 0$, which implies that $\mathcal{Y}(\Omega,\Sigma)=0$.

Summing up, we have that if there is $b<0$ such that $\mathcal{Y}^{q,r}_b(\Omega,\Sigma)>0$, $\mathcal{Y}^{q,r}_0(\Omega,\Sigma)>0$, while if there is $b'>0$ $\mathcal{Y}^{q,r}_{b'}(\Omega,\Sigma)=0$, $\mathcal{Y}^{q,r}_0(\Omega,\Sigma)=0$, and we can't have both at the same time. So the sign of $\mathcal{Y}^{q,r}_b(\Omega,\Sigma)$ does not depend on $b$.

Now, fixing $b=0$, by definition $\mathcal{Y}^{q,r}_0(\Omega,\Sigma)$ does not depend on $r$, and what we said about $b$ holds for $r$ as well.

To prove that the sign of the relative Yamabe invariant does not depend on $q$ either we will also argue that the sign of $\mathcal{Y}^{q,r}_0(\Omega,\Sigma)$ is independent of $q$. If $q>q'$, there exists $C_q>0$ such that $\|\psi\|_{L^{q'}(\Omega)}\leq C_q\|\psi\|_{L^q(\Omega)}$. On the other hand, if $\psi\in B_0^{q,r}(\Omega, \Sigma)$, let $k>0$ be such that $k\psi\in B_0^{q',r}(\Omega, \Sigma)$, so
$$
\begin{array}{cccl}
    &1 &=&k^{q'}\|\psi\|^{q'}_{L^{q'}(\Omega)}  \\
     &&\leq & k^{q'}C^q_{q}\|\psi\|^q_{L^q(\Omega)}\\
     &&=& k^{q'}C^q_{q}\\
\end{array}
$$
and that implies
$$
 k^{q'}=\frac{1}{\|\psi\|_{L^{q'}(\Omega)}^{q'}}\geq\frac{1}{C^q_q}.
$$
Hence, $E(k\psi)=k^2E(\psi)\geq\frac{E(\psi)}{C^{2q}_q}$ and $\mathcal{Y}_0^{q,r}(\Omega,\Sigma)>0$ implies $\mathcal{Y}_0^{q',r}(\Omega, \Sigma)>0$. 

Now, let $\mathcal{Y}^{q',r}_0(\Omega,\Sigma)>0$ and assume $\mathcal{Y}^{q,r}_0(\Omega,\Sigma)=0$. Then we have $\mathcal{Y}^{q',r}_1(\Omega,\Sigma)=\epsilon>0$ and $\mathcal{Y}^{q,r}_1(\Omega,\Sigma)=0$.

Let $\left\{\psi_m\right\}_m\subset B_1^{q,r}(\Omega,\Sigma)$ be a minimizing sequence and $k_m>0$ be such that $k_m\psi_m\in B_1^{q',r}(\Omega,\Sigma)$. So $E(k_m\psi_m)=k_m^2E(\psi_m)\geq \epsilon$ implies $k_m\to \infty$. But $k_m$ satisfies $k_m^{q'}||\psi_m||_{L^{q'}(\Omega)}^{q'}+k_m^r\|\Tr \psi_m\|^r_{L^r(\Sigma)}=1$, thus $\|\psi_m\|_{L^{q'}(\Omega)}\to 0$ and $\|\Tr\psi_m\|_{L^r(\Sigma)}\to 0$. Also, $q'\geq 2,r\geq 2$, so
\begin{equation}
\label{eq:L2norm_and_boundary_to_zero}
    \begin{cases}
    \|\psi_m\|_{L^2(\Omega)}\to 0\\
    \|\Tr \psi_m\|_{L^2(\Sigma)}\to 0
    \end{cases}.
\end{equation}

On the other hand, $\|\Tr\psi_m\|_{L^r(\Sigma)}\to 0$ implies $\|\psi_m\|_{L^q(\Omega)}\to 1$ and thus, since $q\leq 2\bar q$ implies there is $K_q>0$ such that $\|\psi_m\|^2_{W^{1,2}(\Omega)}\geq \frac{\|\psi_m\|_{L^q(\Omega)}}{K_q}$:
\begin{equation}
\label{eq:bound_nabla}
    \lim \|\psi_m\|^2_{W^{1,2}(\Omega)}\geq \frac{1}{K_q}\Rightarrow \lim \|\nabla\psi_m\|^2_{L^2(\Omega)}\geq \frac{1}{K_q}
\end{equation}

Finally, since $E(\psi_m)\to 0$, item (c) of theorem \ref{t:generalities_y} implies $\|\nabla \psi_m\|_{L^2(\Omega)}$ is bounded and then, along with (\ref{eq:L2norm_and_boundary_to_zero}), item (a) of the same theorem implies
\begin{equation}
    \lim \left|\int_\Omega R\psi_m^2dV_g\right|=\lim\left|\int_\Sigma H(\Tr\psi_m)^2d\sigma_g\right|=0
\end{equation}
thus, by definition of $E$:
\begin{equation}
    \lim \|\nabla\psi_m\|^2_{L^2(\Omega)}=\lim E(\psi_m)=0
\end{equation}
which contradicts (\ref{eq:bound_nabla}). So $\mathcal{Y}^{q,r}_0(\Omega,\Sigma)>0$ and the sign of $\mathcal{Y}^{q,r}_b(\Omega,\Sigma)$ does not depend on $q$.
\end{proof}

We finish this section by proving some general properties of $\Yam_b^{q,r}(\Omega,\Sigma)$ that will not be necessary for the rest of the paper.

\begin{lemma}[Monotonicity]
\label{l:monotonicity}
If $\Omega_1\subset\Omega_2$ and $\Sigma_1\subset\Sigma_2$, then we have
$$
\Yam^{q,r}_b(\Omega_1,\Sigma_1)\geq\Yam^{q,r}_b(\Omega_2,\Sigma_2).
$$
\begin{proof}
It follows straight from the definition since we expand the set of test functions $\varphi$.
\end{proof}
\end{lemma}

\begin{lemma}[Continuity from above]
\label{l:continuity_from_above}
Let $\Omega\subset M$, $\Sigma\subset \partial M$, $q$, $r$ and $b$ as in theorem \ref{t:subcrit}. Let $\Omega_k$ be a decreasing sequence of relatively open sets and let $\Sigma_k$ be a decreasing sequence of relatively closed sets such that
\begin{equation*}
    \begin{cases}
    \bigcap_k \Omega_k=\Omega\\
    \bigcap_k \Sigma_k=\Sigma
    \end{cases}.
\end{equation*}
\end{lemma}
Then we have
\begin{equation}
    \lim_k \Yam^{q,r}_b(\Omega_k,\Sigma_k)=\Yam^{q,r}_b(\Omega,\Sigma).
\end{equation}
\begin{proof}
From monotonicity, $(\Omega_k\cup\Sigma_k)\supset(\Omega_{k+1}\cup\Sigma_{k+1})$ implies $\Yam^{q,r}_b(\Omega_k, \Sigma_k)\leq \Yam^{q,r}_b(\Omega_{k+1}, \Sigma_{k+1})$, so the sequence of the $\Yam_k:=\Yam^{q,r}_b(\Omega_k,\Sigma_k)$ is non-decreasing. In addition, again by monotonicity, $\Yam:=\Yam^{q,r}_b(\Omega,\Sigma)\geq \Yam_k$, for all $k$. Hence, we can define $\Lambda=\displaystyle\lim_k \Yam_k\geq \Yam_k$, for all $k$. We also have $\Yam\geq \Lambda$.

If $\Lambda=+\infty$, $\Yam=+\infty$ as well and the result holds automatically.

On the other hand, assume $\Lambda$, and thus the $\Yam_k$, are all finite. So, by Theorem \ref{t:subcrit}, for each $k$ there is $u_k\in B^{q,r}_b(\Omega_k,\Sigma_k)$ satisfying $E(u_k)=\Yam_k\leq \Lambda$. Then, by lemmata \ref{t:generalities_y} and \ref{l:estimates_L2norms}, there is $K>0$ such that $\|u_k\|_{W^{1,2}(M)}\leq K,\forall k$. Now, all $u_k$ are in $W^{1,2}(\Omega_1,\Sigma_1)$, implying that the sequence of the $u_k$ is bounded in $W^{1,2}(\Omega_1)$ and hence, in an argument similar to the one in theorem \ref{t:subcrit}, there is $u\in W^{1,2}(\Omega_1,\Sigma_1)$ such that:
\begin{enumerate}[(i)]
    \item $u_k\rightharpoonup u$ in $W^{1,2}(\Omega_1,\Sigma_1)$,
    \item $u_k\to u$ in $L^q(\Omega_1)$,
    \item $\Tr u_k\to \Tr u$ in $L^r(\Sigma_1)$,
    \item $u_k\to u$ a.e. in $\Omega_1$,
    \item $\Tr u_k\to \Tr u$ a.e. in $\Sigma_1$.
\end{enumerate}

By (iv) and (v), $u\in W^{1,2}(\Omega,\Sigma)$. In fact, by (ii) and (iii), $u\in B^{q,r}_b(\Omega,\Sigma)$. So $E(u)\geq \Yam$. But, just as in the proof of theorem \ref{t:subcrit}, we have $E(u)\leq \liminf E(u_k)=\Lambda$. Thus we conclude that $\Lambda=\Yam$.
\end{proof}

As a remark, the result is not true if we drop the hypothesis on the indices being subcritical, as explained after the proof of Lemma 3.14 in \cite{DM15}.

\section{Eigenvalue of the Laplacian}
\label{sec:eigenvalue}

To study the more general problem of prescribing scalar curvature and mean extrinsic curvature we try to adapt the strategy of the paper \cite{DM15} to the case of manifolds with boundary. In order to do so we have to study a specific instance of $\Yam^{q,r}_b(\Omega, \Sigma)$ not covered by the results in the previous section, with $q=r=2$, which corresponds to the first eigenvalue of the Laplacian relative to the pair $(\Omega,\Sigma)$.

First we define the Raileigh quotient for functions in $W^{1,2}(M,\partial M)$ as

\begin{equation}
    \label{eq:rayleigh}
    Q_g(\varphi)=\frac{E(\varphi)}{||\varphi||_{L^2(\Omega)}^2+||\gamma\varphi||^2_{L^2(\Sigma)}}
\end{equation}
and the first eigenvalue of the Laplacian for the set $\Omega\cup\Sigma$ as
\begin{equation}
    \label{eq:eigenvalue_laplacian}
    \lambda^{q,r}_b(\Omega,\Sigma):=\inf_{B^{q,r}_{b}(\Omega,\Sigma)}Q_g(\varphi)
\end{equation}

Notice that, since $Q_g(k\varphi)=Q_g(\varphi)$, for all $k>0$, Lemma \ref{l:polynomial} implies that $\lambda^{q,r}_b(\Omega,\Sigma)$ does not really depend on $q$, $r$ or $b$, so we can drop the indices.

\begin{proposition}
\label{c:minimizer_eigenvalue_Laplacian}
If $(M, g),\ \Omega,\ \Sigma$ are as before, there is $u\in W^{1,2}(\Omega,\Sigma)$ such that
$$
\lambda(\Omega,\Sigma)=Q_g(u).
$$
\begin{proof}
We want to be able to simply take $q=r=2$, $b=1$ in theorem \ref{t:subcrit}. Checking the proof, $q>r$ is a relevant hypothesis just to show that a minimizing sequence $\{\phi_i\}_i\subset B_1^{2,2}(\Omega, \Sigma)$ for $E(\phi_i)\to \lambda(\Omega, \Sigma)$ will be bounded in $W^{1,2}(\Omega, \Sigma)$. Now, if $\phi_i\in B_1^{2,2}(\Omega,\Sigma)$, $\|\phi_i\|_{L^2(\Omega)}\leq 1$. Also, Proposition \ref{t:generalities_y} (a) does not require $q>r$, so for any $\epsilon>0$, there is $ K_\epsilon>0$ such that
$$
\begin{array}{ccl}
    \int_{\Omega} |\nabla \phi_i|^2dV_g &=& E(\phi_i)-\frac{n-2}{4(n-1)}\int_\Omega R\phi_i^2dV_g-\frac{n-2}{2}\int_\Sigma H(\gamma\phi_i)^2d\sigma_g  \\
     & \leq& E(\phi_i)+\epsilon \|\phi_i\|^2_{W^{1,2}(\Omega)}+K_\epsilon\|\phi_i\|^2_{L^2(\Omega)}+\epsilon\|\Tr\phi_i\|^2_{W^{\frac{1}{2}, 2}(\Sigma)}+K_{\epsilon}\|\Tr\phi_i\|^2_{L^2(\Sigma)}\\
     &\leq& E(\phi_i)+\epsilon(1+C)\int_\Omega |\nabla\phi_i|^2dV_g+\epsilon+\epsilon C+2K_\epsilon,
\end{array}
$$
with $C$ the constant associated to the trace inequality. So, if we choose $\epsilon$ small enough, for any $i$
\begin{equation}
    \|\phi_i\|^2_{W^{1,2}(\Omega)}\leq \frac{1}{1-\epsilon(1+C)}\left(E(\phi_i)+\epsilon(1+C)+2K_\epsilon\right)+1
\end{equation}
and the proof follows as in theorem \ref{t:subcrit}.
\end{proof}
\end{proposition}

Instead of dealing with the relative Yamabe invariant of a set we want to work with its linear counterpart, the first eigenvalue of the Laplacian. For our purposes, we only need $\lambda(\Omega, \Sigma)$ to coincide with $\mathcal{Y}^{q,r}_b(\Omega,\Sigma)$ in sign. In views of Theorem \ref{t:yclass}, it is enough to prove that the signs agree for $b=1$.

\begin{theorem}
\label{t:equal_signs}
If $q>r$, we have
$$
\sig(\lambda(\Omega,\Sigma))=\sig(\mathcal{Y}^{q,r}_1(\Omega,\Sigma)).
$$
\begin{proof} Let us omit the sets to simplify notation.

i. $\mathcal{Y}^{q,r}_1<0$ if, and only if, there is $\varphi\in B^{q,r}_1(\Omega, \Sigma)$ such that $E(\varphi)<0$, which is itself true if, and only if, $\lambda<0$.

ii. Assume $\mathcal{Y}^{q,r}_1=\epsilon>0$. Since $r,q\geq 2$, there is $C_{q,r}>0$ such that $||\varphi||^2_{L^2(\Omega)}\leq C_{q,r}||\varphi||^q_{L^q(\Omega)}$ and $||\gamma\varphi||^2_{L^2(\Sigma)}\leq C_{q,r}||\gamma\varphi||^r_{L^r(\Sigma)},\forall \varphi\in B^{q,r}_1(\Omega,\Sigma)$. So, if $\varphi\in B^{q,r}_1(\Omega,\Sigma)$:
$$
||\varphi||^2_{L^2(\Omega)}+||\gamma\varphi||^2_{L^2(\Sigma)}\leq C_{q,r}.1\Rightarrow \frac{E(\varphi)}{||\varphi||^2_{L^2(\Omega)}+||\gamma\varphi||^2_{L^2(\Sigma)}}\geq \frac{E(\varphi)}{C_{q,r}},
$$
and $\lambda\geq \frac{\mathcal{Y}^{q,r}_1}{C_{q,r}}=\frac{\epsilon}{C_{q,r}}>0$.

iii. Assume $\lambda>0$ and $\mathcal{Y}^{q,r}_1\leq 0$, hence $\mathcal{Y}_0^{q,r}\leq 0$. Let $\varphi_k$ be a minimizing sequence for $\mathcal{Y}^{q,r}_0$. Then we can assume, up to a subsequence, $E(\varphi_k)<\frac{1}{k},\forall k$.

So we have that, since $\lambda>0$:
\begin{equation}
\label{eq:norm_to_zero}
E(\varphi_k)=Q_g(\varphi_k)\left(||\varphi_k||^2_{L^2(\Omega)}+||\gamma\varphi_k||^2_{L^2(\Sigma)}\right)<\frac{1}{k}\Rightarrow \lim_{k\to\infty}\left(||\varphi_k||^2_{L^2(\Omega)}+||\gamma\varphi_k||^2_{L^2(\Sigma)}\right)=0.    
\end{equation}

Also
\begin{equation}
\label{eq:energy_bound}
    E(\varphi_k)\leq \frac{1}{k}\Rightarrow ||\nabla\varphi_k||^2_{L^2(\Omega)}\leq \frac{1}{k}-\frac{n-2}{4(n-1)}\int_\Omega R \varphi_k^2dV_g-\frac{n-2}{2}\int_\Sigma H (\gamma \varphi_k)^2d\sigma_g
\end{equation}
and $||\varphi_k||_{L^q(\Omega)}= 1,\ q\leq 2\bar q$. As a consequence, there is $K_q>0$ such that $||\varphi_k||^2_{W^{1,2}(\Omega)}\geq \frac{1}{K_q}$ and, by (\ref{eq:norm_to_zero})
\begin{equation}\label{eq:lim_H_norm}
    \lim_{k\to\infty}||\nabla\varphi_k||^2_{L^2(\Omega)}\geq \frac{1}{K_q}.
\end{equation}
Turning back to equations (\ref{eq:R_estimate_L2}) and (\ref{eq:H_estimate_L2})
we see that, given $\epsilon>0$, there exists $K_\epsilon>0$ such that
\begin{equation}
\label{eq:first_estimate_integral_curvatures}
    \begin{cases}
-\frac{n-2}{4(n-1)}\int_\Omega R \varphi^2dV_g\leq \frac{n-2}{4(n-1)}\epsilon ||\nabla\varphi_k||^2_{L^2(\Omega)}+\frac{n-2}{4(n-1)}K_\epsilon ||\varphi_k||^2_{L^2(\Omega)}\\
-\frac{n-2}{2}\int_\Sigma H (\gamma\varphi_k)^2d\sigma_g\leq \frac{n-2}{2}\epsilon ||\nabla\varphi_k||^2_{L^2(\Omega)}+\frac{n-2}{2}K_\epsilon ||\gamma\varphi_k||^2_{L^2(\Sigma)}.
\end{cases}
\end{equation}
But also, by (c) in lemma \ref{t:generalities_y}
$$
||\nabla\varphi_k||^2_{L^2(\Omega)}\leq KE(\varphi_k)+C.
$$
Plugging that into (\ref{eq:first_estimate_integral_curvatures}) we get
$$
\begin{cases}
-\frac{n-2}{4(n-1)}\int_\Omega R \varphi^2dV_g\leq \frac{n-2}{4(n-1)}\epsilon \left(KE(\varphi_k)+C\right)+\frac{n-2}{4(n-1)}K_\epsilon ||\varphi_k||^2_{L^2(\Omega)}\\
-\frac{n-2}{2}\int_\Sigma H (\gamma\varphi_k)^2d\sigma_g\leq \frac{n-2}{2}\epsilon \left(KE(\varphi_k)+C\right)+\frac{n-2}{2}K_\epsilon ||\gamma\varphi_k||^2_{L^2(\Sigma)}
\end{cases}
$$
and if we take the limit in (\ref{eq:energy_bound}) and use the result (\ref{eq:norm_to_zero}) we end up finding that there is a constant $L=(KE(\varphi_k)+C)\left(\frac{n-2}{4(n-1)}+\frac{n-2}{2}\right)$ such that
$$
\lim_{k\to\infty}||\nabla\varphi_k||^2_{L^2(\Omega)}\leq \epsilon L,
$$
for all $\epsilon>0$. Contradiction.

iv. Gathering the results i, ii and iii we find that $\mathcal{Y}^{q,r}_1=0$ if, and only if, $\lambda=0$.
\end{proof}
\end{theorem}

\section{The Prescribed Scalar Curvature and Mean Curvature Problem}
\label{sec:prescribed}

Let's define the following functional over $W^{1,2}(M,\partial M)$
\begin{equation}
\label{eq:functional_prescribed_curvatures}
    F_{q,r}(u)=E(u)-\int_M \frac{n-2}{2q(n-1)} R'|u|^{q}dV_g-\int_{\partial M} \frac{n-2}{r} H'|\gamma u|^rd\sigma_g
\end{equation}
and prove some of its important properties. But first we need a technical lemma.
\begin{lemma}
\label{l:unmatched_exponents}
Let $A,B,C>0$ be real constants, $x,y,z\geq 0$ be positive numbers and assume
\begin{equation}
    y+z>Cx.
\end{equation}
If $q,r,p$ are exponents satisfying $q\geq q_0>1$, $r\geq r_0>1$ and $r_0,q_0>p>1$, there is $J(q_0, r_0, p)>0$ such that if $x\geq L$
$$
f(y,z)=Ay^q+Bz^r\geq x^p.
$$
\begin{proof}
If $Ay^q\geq x^p$, there is nothing to do. So assume $y<A^{-\frac{1}{q}}x^{\frac{p}{q}}$. Then
\begin{equation}
    z-B^{-\frac{1}{r}}x^{\frac{p}{r}}>Cx-y-B^{-\frac{1}{r}}x^{\frac{p}{r}}>Cx-A^{-\frac{1}{q}}x^{\frac{p}{q}}-B^{-\frac{1}{r}}x^{\frac{p}{r}}
\end{equation}
Since $\frac{p}{q}, \frac{p}{r}<1$, there is $L>0$ such that if $x\geq L$ then
\begin{equation}
    Cx-A^{-\frac{1}{q}}x^{\frac{p}{q}}-B^{-\frac{1}{r}}x^{\frac{p}{r}}>0
\end{equation}
and hence $z>B^{-\frac{1}{r}}x^{\frac{p}{r}}$ and $Bz^r>x^p$.
\end{proof}
\end{lemma}

For simplicity, for the rest of this section, we will use the notation
$$
\|u\|^2_{L^2(M,\partial M)}=\|u\|^2_{L^2(M)}+\|\gamma u\|^2_{L^2(\partial M)},
$$
which defines a norm in $W^{1,2}(M)$.

Also, we will denote $Z:=\{s\in M| R'(s)=0\}$ and $Z_\partial=\{p\in \partial M| H'(p)=0\}$ the zero sets of $R'$ and $H'$ respectively.

Finally, we can prove our first important result.

\begin{lemma}[Coercivity]
\label{l:coercivity}
Let  $q_0,$ $r_0$ be subcritical (as in theorem \ref{t:subcrit}), $R'\leq 0$, $H'\leq 0$.

If $(Z,Z_\partial)$ is Yamabe positive (i.e. $\mathcal{Y}^{p,q}_b(Z, Z_\delta)>0$),  $\forall B\in\mathbb{R},\  \exists K(q_0,r_0,B)>0$ such that if  $q,$ $r$ are subcritical, $q\geq q_0,\ r\geq r_0,\ q>r$, $u\in W^{1,2}(M,\partial M)$ and $||u||_{L^2(M,\partial M)}\geq K$, then $F_{q,r} (u)\geq B$.
\begin{proof}
Adapting the proof of proposition 4.5 of \cite{DM15}.

For each $\epsilon>0$, let's define
\begin{equation}
\begin{array}{ccl}
    A_\epsilon &=&\left\{u\in W^{1,2}(M, \partial M)\left| \int_M |R'|u^2dV_g+\int_{\partial M}|H'|(\Tr u)^2d\sigma_g\right.\right.\\
     & &\left.\left.\leq \epsilon ||u||^2_{L^2(M,\partial M)}\left(\int_M|R'|dV_g+\int_{\partial M}H'd\sigma_g \right)\right.\right\}.
\end{array}
\end{equation}
Now, since $(Z, Z_\partial)$ is Yamabe positive, $\lambda(Z, Z_\partial)>0$. So we can fix $L\in(0, \lambda(Z, Z_\partial))$.

\textbf{Claim:} There is $\epsilon_0<1$ such that if $u\in A_{\epsilon_0}$
\begin{equation}
\label{eq:L_condition}
    E(u)\geq L||u||^2_{L^2(M,\partial M)}.
\end{equation}
Assume that is not true. So if $\epsilon_k$ is a sequence such that $\epsilon_k\to 0$, we can choose $v_k\in A_{\epsilon_k}$, $||v_k||_{L^2(M, \partial M)}=1$ violating equation (\ref{eq:L_condition}):
\begin{equation}
    E(v_k)<L,\forall k.
\end{equation}
Putting this together with item 3 of lemma \ref{t:generalities_y} we have that $\{v_k\}$ is a bounded sequence in $W^{1,2}(M)$. So, just as in the proof of theorem \ref{t:subcrit}, there is $v\in W^{1,2}(M)$ such that $v_k\rightharpoonup v$ in $W^{1,2}(M)$ and $v_k\to v$ in $L^2(M,\partial M)$, and hence $||v||_{L^2(M,\partial M)}=1$.

Also, as in the proof of theorem \ref{t:subcrit}, changing $R$ by $R'$,  we have that
\begin{equation}
\int_M |R'|v_k^2 dV_g\to \int_M |R'|v^2 dV_g    
\end{equation}
and
\begin{equation}
\int_{\partial M} |H'|(\Tr v_k)^2 d\sigma_g\to \int_{\partial M} |H'|(\Tr v)^2 d\sigma_g
\end{equation}

But, by choice of the $v_k$'s
\begin{equation}
    0\leq \int_M |R'|v_k^2dV_g+\int_{\partial M}|H'|(\gamma v_k)^2d\sigma_g\leq \epsilon_k ||v_k||^2_{L^2(M,\partial M)}\left(\int_M|R'|dV_g+\int_{\partial M}|H'|d\sigma_g\right)\to 0
\end{equation}
So
\begin{equation}
    \int_M |R'|v^2dV_g+\int_{\partial M}|H'|(\Tr v)^2d\sigma_g=0
\end{equation}
So $v\equiv 0$ outside $Z\cup Z_\partial$ and $v\in W^{1,2}(Z, Z_\partial)$. But, since $v_k\rightharpoonup v$ in $W^{1,2}$, $v_k\to v$ in $L^2$ and from the last equality
$$
E(v)\leq E(v_k)\leq L< \lambda(Z, \emptyset).
$$
Contradiction.

Now we divide the proof of the theorem itself in two cases, taking $\epsilon_0$ as in the claim:

\textbf{Case 1: ($u\not\in A_{\epsilon_0}$)} If $u\not\in A_{\epsilon_0}$
\begin{equation}
\label{eq:Condition_case_1}
    \int_M |R'|u^2dV_g+\int_{\partial M}|H'|(\Tr u)^2d\sigma_g> \epsilon_0 ||u||^2_{L^2(M,\partial M)}\left(\int_M|R'|dV_g+\int_{\partial M}|H'|d\sigma_g\right)
\end{equation}
Now, if $R', H'\leq 0$
\begin{equation}
\label{eq:Fqr_with_negative_curvature}
    F_{q,r}(u)=E(u)+\int_M \frac{n-2}{2q(n-1)}|R'||u|^qdV_g+\int_{\partial M}\frac{n-2}{r}|H'||\gamma u|^rd\sigma_g
\end{equation}
Now, using H\"older inequality on the higher power terms
\begin{equation}
\label{eq:estimate_R'_q}
    \int_M |R'||u|^2dV_g=\int_M |R'|^{1-\frac{2}{q}}|R'|^{\frac{2}{q}}|u|^2dV_g\leq \left[\int_M|R'|dV_g\right]^{1-\frac{2}{q}}\left[\int_M |R'||u|^q dV_g\right]^{\frac{2}{q}}
\end{equation}
and similarly
\begin{equation}
\label{eq:estimate_H'_r}    
    \int_{\partial M}|H'||\Tr u|^2d\sigma_g\leq \left[\int_{\partial M}|H'|d\sigma_g\right]^{1-\frac{2}{r}}\left[\int_{\partial M}|H'||\gamma u|^rd\sigma_g\right]^{\frac{2}{r}}
\end{equation}

On the other hand, from lemma \ref{t:generalities_y} (a) and (b), given $\eta>0$, there are $D_1(\eta), D_2(\eta)$, independent of $q$ or $r$ such that
\begin{equation*}
\begin{array}{ccl}
    E(u) & = &\int_M |\nabla u|^2dV_g+\frac{n-2}{4(n-1)}\int_M Ru^2dV_g+\frac{n-2}{2}\int_{\partial M} H(\Tr u)^2d\sigma_g \\
     & \geq & \int_M|\nabla u|^2dV_g-\frac{n-2}{4(n-1)}\left(\eta \|u\|^2_{W^{1,2}(M)}+D_1\|u\|^2_{L^2(M)}\right)\\
     &&-\frac{n-2}{2}\left( \eta \|\Tr u\|^2_{W^{1,2}(\partial M)}+D_2\|\Tr u\|^2_{L^2(\partial M)}\right)\\
     &\geq& \left( 1-\left(\frac{n-2}{4(n-1)}-\frac{n-2}{2}D\right)\eta \right)\int_M|\nabla u|^2dV_g-\left(\frac{n-2}{4(n-1)}(\eta+D_1)+\frac{n-2}{2}(\eta D)\right)\|u\|^2_{L^2(M)}\\
     &&-\frac{n-2}{2}D_2\|\Tr u\|^2_{L^2(\partial M)}
\end{array}
\end{equation*}
with $D$ the constant associated to the trace inequality. By choosing $\eta$ small enough that the coefficient of $\int_M\|\nabla u\|^2dV_g$ is positive and $D_3$ the largest value between $\frac{n-2}{4(n-1)}(\eta+D_1)+\frac{n-2}{2}\eta D$ and $\frac{n-2}{2}D_2$ we have
\begin{equation}
    E(u)\geq -D_3\|u\|^2_{L^2(M,\partial M)}.
\end{equation}
Plugging this together with equations (\ref{eq:estimate_R'_q}) and (\ref{eq:estimate_H'_r}) in (\ref{eq:Fqr_with_negative_curvature}) we have
\begin{equation}
    F_{q,r}(u)\geq -D_3\|u\|^2_{L^2(M,\partial M)}+A_1(q)\left[\int_M |R'||u|^2dV_g\right]^{\frac{q}{2}}+A_2(r)\left[\int_{\partial M}|H'||\Tr u|^2d\sigma_g\right]^{\frac{r}{2}}
\end{equation}
with $A_1(q)=\frac{n-2}{2q(n-1)}\left[\int_{M}|R'|dV_g\right]^{1-\frac{q}{2}}$ and $A_2(r)=\frac{n-2}{r}\left[\int_{\partial M}|H'|d\sigma_g\right]^{1-\frac{r}{2}}$ positive constants that depend on $q$ and $r$ respectively. Since both $A_1(q)$ and $A_2(r)$ are continuous on $q$ and $r$, we can define $A_1:=\min\{A_1(q)|q\in[q_0, 2\bar q]\}$ and $A_2:=\min\{A_2(r)|r\in [r_0, \bar q+1]\}$. Notice that they cannot be both zero otherwise inequality (\ref{eq:Condition_case_1}) is not satisfied. We can then drop the dependency on $q$ and $r$ of the constants in the previous inequality and have
\begin{equation}
\label{eq:unrefined_F_estimate}
    F_{q,r}(u)\geq -D_3\|u\|^2_{L^2(M,\partial M)}+A_1\left[\int_M |R'||u|^2dV_g\right]^{\frac{q}{2}}+A_2\left[\int_{\partial M}|H'||\Tr u|^2d\sigma_g\right]^{\frac{r}{2}}
\end{equation}

Now, $A_2=0$ if, and only if, $H'\equiv 0$. In this case the right-hand side of inequality (\ref{eq:unrefined_F_estimate}) becomes
\begin{equation}
    F_{q,r}(u)\geq -D_3(\|u\|^2_{L^2(M,\partial M)})+A_1\left[\int_{\partial M}|R'| u^2dV_g\right]^{\frac{q}{2}}
\end{equation}
and by condition (\ref{eq:Condition_case_1}) we have

\begin{equation}
\label{eq:lower_bound_F_negative}
    F_{q,r}\geq -D_3\|u\|^2_{L^2(M,\partial M)}+A_1\left[\epsilon_0 \int_M|R'|dV_g\right]^{\frac{q}{2}}||u||^q_{L^2(M,\partial M)}
\end{equation}

As said, $\int_M|R'|dV_g>0$, thus the coefficient of $||u||^q_{L^2(M,\partial M)}$ is larger than zero and, since $q\geq q_0>2 $, there is $K$ satisfying the hypothesis of the theorem. In fact, if $K$ satisfies the hypothesis for $q=q_0$, the same $K$ satisfies for $q>q_0$. A similar thing happens if $A_2=0$ instead.

Finally, if we assume $A_1, A_2>0$, since $\frac{q}{2}>\frac{r}{2}>1$, by lemma \ref{l:unmatched_exponents}, there is $p>1$ and $J>0$, independent of $q, r$ such that $\|u\|^2_{L^2(M,\partial M)}\geq J$ implies, as a consequence of condition (\ref{eq:Condition_case_1})
\begin{equation}
    A_1\left[\int_M |R'||u|^2dV_g\right]^{\frac{q}{2}}+A_2\left[\int_{\partial M}|H'||\Tr u|^2d\sigma_g\right]^{\frac{r}{2}}\geq \|u\|^{2p}_{L^2(M,\partial M)}
\end{equation}
and replacing in (\ref{eq:unrefined_F_estimate})
\begin{equation}
\label{eq:lower_bound_F_non_zero_curvatures}
     F_{q,r}(u)\geq -D_3\|u\|^2_{L^2(M,\partial M)}+\|u\|^{2p}_{L^2(M,\partial M)}
\end{equation}
and because $p>1$ we can choose $K>J$ satisfying the theorem.

\textbf{Case 2: ($u\in A_{\epsilon_0}$)} In that case, if $R', H'\leq 0$:
\begin{equation}
\label{eq:lower_bound_F_zero}
    F_{q,r}(u)\geq E(u)\geq L||u||^2_{L^2(M,\partial M)}.
\end{equation}
And the result follows immediately.
\end{proof}
\end{lemma}

Notice that in the case of inequality (\ref{eq:lower_bound_F_zero}), $F_{q,r}\geq 0$, while in the case of inequalities (\ref{eq:lower_bound_F_negative}) and (\ref{eq:lower_bound_F_non_zero_curvatures}) the right hand sides also have lower bounds, independent of $u$. Putting all together we see that $F_{q,r}$ has a lower bound independent of $u$. It is natural, then, to look for minimizers.

\begin{proposition}
Under the hypothesis of the previous lemma, fixing $q, r$ subcritical, assume $R'$ and $H'$ are bounded. Assume also that
\begin{equation}
\label{eq:condition_negativity_of_curvatures}
    \frac{n-2}{4(n-1)}\int_M RdV_g+\frac{n-2}{2}\int_{\partial M}Hd\sigma_g<0
\end{equation}
Then there is $u_{q,r}>0$ a function that minimizes $F_{q,r}$ and is a weak solution to
\begin{equation}
\label{e:F_subcritical_problem}
\begin{cases}
    -\Delta u_{q,r}+\frac{n-2}{4(n-1)}R u_{q,r}=\frac{n-2}{4(n-1)}R' u_{q,r}^{q-1},\ in\ \Omega\\
\Tr\partial_\nu u_{q,r}+\frac{n-2}{2}H\Tr u_{q,r}=\frac{n-2}{2}H'(\Tr u_{q,r})^{r-1},\ in\ \Sigma
\end{cases}.
\end{equation}
\end{proposition}

\begin{proof}
Let $\{u_k\}_k$ be a minimizing sequence for $F_{q,r}$ in $W^{1,2}(M,\partial M)$. Since $F_{q,r}(|u|)=F_{q,r}(u)$, we can assume $u_k\geq 0$ for each $k$.

Now, fixing a constant function $\tilde u\equiv A\in W^{1,2}(M,\partial M)$ we have
\begin{equation}
\label{eq:defining_B}
    F_{q,r}(\tilde u)= A^2\left[\frac{n-2}{4(n-1)}\int_M R  dV_g+\frac{n-2}{2}\int_{\partial M}H d\sigma_g\right]-A^q \frac{2}{q}\int_M R' dV_g-A^{r}\frac{2}{r}\int_{\partial M}H' d\sigma_g:=B
\end{equation}
So, since $\{u_k\}_k$ is minimizing, we can assume $F_{q,r}(u_k)\leq B, \forall k$ and, by the previous lemma, $\exists K_B>0$ such that $\|u_k\|_{L^2(M,\partial M)}<K_B$. Now, as seen in the proof of the previous lemma, $F_{q,r}(u)\geq E(u),\forall u$, so this, together with item a of lemma \ref{t:generalities_y}, shows that 
$$
\begin{array}{ccl}
    \|u_k\|^2_{W^{1,2}(M)} &=& E(u_k)-\frac{n-2}{4(n-1)}\int_MRu_k^2dV_g-\frac{n-2}{2}\int_{\partial M}H(\Tr u_k)^2d\sigma_g+\|u_k\|^2_{L^2(M)} \\
    &\leq & B+\frac{n-2}{4(n-1)}|\int_MRu_k^2dV_g|+\frac{n-2}{2}|\int_{\partial M}H(\Tr u_k)^2d\sigma_g| +K_B^2\\
    &\leq&  B+\frac{n-2}{4(n-1)}\left(\epsilon\|u_k\|^2_{W^{1,2}(M)}+K_\epsilon \|u_k\|^2_{L^2(M)}\right)\\
    &&+\frac{n-2}{2}\left(\epsilon\|\Tr u_k\|^2_{W^{\frac{1}{2}, 2}(\partial M)}+K_\epsilon \|\Tr u_k\|^2_{L^2(\partial M)}\right) +K_B^2\\
    &\leq& B+K_B^2\left(1+\frac{n-2}{4(n-1)}K_\epsilon+\frac{n-2}{2}K_\epsilon\right)+\epsilon\left(\frac{n-2}{4(n-1)}+\frac{n-2}{2}C\right)\|u_k\|^2_{W^{1,2}(M)}
\end{array}
$$
$C$ the constant given by the trace inequality. So, choosing $\epsilon$ small enough we finally have
\begin{equation}
\label{eq:lower_bound_H_1_norm}
\|u_k\|^2_{W^{1,2}(M)}\leq \frac{B+K_B^2\left(1+\frac{n-2}{4(n-1)}K_\epsilon+\frac{n-2}{2}K_\epsilon\right)}{1-\epsilon\left(\frac{n-2}{4(n-1)}+\frac{n-2}{2}C\right)}
\end{equation}
so $\{u_k\}_k$ is a bounded sequence in $W^{1,2}(M)$ and we can use the same argument to say that there is $u_{q,r}$ in $W^{1,2}(M)$ such that:
\begin{enumerate}[(i)]
    \item $u_k\rightharpoonup u_{q,r}$ in $W^{1,2}(M)$;
    \item $u_k\to u_{q,r}$ in $L^{p}(M)$, $p\in [2, 2\bar q)$;
    \item $\gamma u_k\to \gamma u_{q,r}$ in $L^{s}(\partial M)$, $s\in [2, \bar q+1)$.
\end{enumerate}
Now, as in the proof of theorem \ref{t:subcrit}, we have that $E(u_{q,r})\leq \liminf{E(u_k)}$. Also, since $R'$, $H'$ are bounded and $u_k\to u_{q,r}$ in $L^{q}(M)$ and $\Tr u_k\to \Tr u_{q,r}$ in $L^r(\partial M)$:
\begin{equation}
    \begin{cases}
    \int_M R'u_k^q dV_g\to \int_M R'u^q_{q,r}dV_g\\
    \int_M H'(\Tr u_k)^r dV_g\to \int_M H'(\Tr u_{q,r})^rdV_g
    \end{cases}
\end{equation}
So $F_{q,r}(u_{q,r})\leq \liminf F_{q,r}(u_k)$ and, since $u_k$ is a minimizing sequence, $u_{q,r}$ is a minimizer of $F_{q,r}$ and hence a weak solution to equation (\ref{e:F_subcritical_problem}).

Now, looking back at equation (\ref{eq:defining_B}), we see that if condition (\ref{eq:condition_negativity_of_curvatures}) is true, since $q>r\geq 2$ we can choose $A$ small enough to guarantee that $B<0$, hence $F_{q,r}(u_{q,r})<0$ and $u_{q,r}\not\equiv 0$. Thus, by lemma B.7 of \cite{HT13}, $u_{q,r}>0$ everywhere. 
\end{proof}

Notice that (\ref{eq:lower_bound_H_1_norm}) gives an uniform bound on $\|u_{q,r}\|_{W^{1,2}(M)}$, which, of course, gives an uniform bound on $\|u_{q,r}\|_{L^{2\bar q}(M)}$. We want to use this fact to get regularity results for the minimizers by resorting to a bootstrap argument, but we need first to increase this regularity a little.

\begin{lemma}
\label{l:extra_regularity} In the context of the previous lemma, there is $Q>2\bar q$ and $C>0$ independent of $q$ and $r$ such that
\begin{equation}
    \|u_{q,r}\|_{L^Q(M)}\leq C.
\end{equation}
\end{lemma}

\begin{proof}
Let $u_{q,r}$ be a minimizer of $F_{q,r}$ as before and define for some $\delta>0$ the auxiliary functions:
\begin{equation}
    \begin{cases}
    w=u_{q,r}^{1+\delta},\\
    \nu=u_{q,r}^{1+2\delta}.
    \end{cases}
\end{equation}
Since $u_{q,r}$ is a weak solution to (\ref{e:F_subcritical_problem}), we can test it against $\nu$ to have, defining $C_\delta=\frac{1+2\delta}{(1+\delta)^2}$:
$$
\begin{array}{ccl}
    C_\delta\|\nabla w\|^2_{L^2(M)}&=&\int_M\langle \nabla u_{q,r}, \nabla \nu \rangle dV_g \\
    &=&\frac{n-2}{4(n-1)}\int_M \left(R'u_{q,r}^{q-1}-Ru_{q,r}\right)\nu dV_g+\frac{n-2}{2}\int_{\partial M} \left(H'(\Tr u_{q,r})^{r-1}-H\Tr u_{q,r} \right)(\Tr\nu) d\sigma_g\\
    &\leq & -\left( \frac{n-2}{4(n-1)}\int_M Ru\nu dV_g+\frac{n-2}{2}\int_{\partial M}H(\gamma u)(\gamma\nu)d\sigma_g\right)\\
    &\leq& \frac{n-2}{4(n-1)}\left\vert\int_M Ru\nu dV_g\right\vert+\frac{n-2}{2}\left\vert\int_{\partial M}H(\gamma u)(\gamma\nu)d\sigma_g\right\vert\\
    &=&\frac{n-2}{4(n-1)}\left\vert\int_M Rw^2 dV_g\right\vert+\frac{n-2}{2}\left\vert\int_{\partial M}H(\gamma w)^2d\sigma_g\right\vert,
\end{array}
$$
the first inequality being true because $R'\leq 0$ and $H'\leq 0$. So, using the estimates from Lemma \ref{t:generalities_y}a we have that for $\epsilon>0$, there is $K_\epsilon>0$ such that
\begin{equation}
   C_\delta \|\nabla w\|^2_{L^2(M)}\leq \epsilon\|w\|^2_{W^{1,2}(M)}+K_\epsilon \|w\|^2_{L_2(M)}+\epsilon \|\Tr w\|^2_{W^{\frac{1}{2}, 2}(\partial M)}+K_\epsilon\|\Tr w\|^2_{L^2(\partial M)}
\end{equation}
and hence, if $L$ is the constant associated to the trace inequality, rearranging the terms we have
\begin{equation}
    (C_\delta-\epsilon-\epsilon L)\|\nabla w\|^2_{L^2(M)}\leq (\epsilon+K_\epsilon+\epsilon L)\|w\|^2_{L^2(M)}+K_\epsilon\|\Tr w\|^2_{L^2(\partial M)}.
\end{equation}
If we choose $\delta$ small enough to guarantee the right hand side is controlled by $\|u_{q,r}\|_{L^{2\bar q}(M)}$, and then $\epsilon>0$ such that the left hand side is positive, we can use the Sobolev inequality to argue there is $C'>0$ such that
\begin{equation}
    \|w\|_{L^{2\bar q}(M)}\leq C'
\end{equation}
and the result follows by taking $Q=2\bar q(1+\delta)$.
\end{proof}

Finally, we can prove the main regularity result.

\begin{lemma}[Regularity]
\label{l:regularity}
Assume $R,\ R'\in L^{\infty}(M)$, $H,\ H'\in W^{s,t}(\partial M)$ with $t>1$, $s\geq 1$ and $st>n-1$. If $u_{q,r}$ is a minimizer of $F_{q,r}$ as in the previous lemma, then $u_{q,r}\in W^{2, \frac{n}{2}}(M)$, $\gamma u_{q,r}\in W^{2, \frac{n-1}{2}}(\partial M)$ and there are constants $C, K$ {\it{independent of q and r}} such that
\begin{equation}
\begin{cases}
\|u_{q,r}\|_{W^{2,\frac{n}{2}}(M)}\leq C\\
\|\Tr u_{q,r}\|_{W^{2,\frac{n-1}{2}}(\partial M)}\leq K
\end{cases}\end{equation}
\end{lemma}
\begin{proof}
Both the regularity and the uniform bounds come from a bootstrap procedure through applications of corollary B.4 of \cite{HT13}. Taking $D=\emptyset$, $\exists C_{p_i}>0$ such that
$$
\begin{array}{ccl}
    C_{p_i}\| u_{q,r}\|_{W^{2, p_i}(M)} &\leq & \|\Delta u_{q,r}\|_{W^{0,p_i}(M)}+\|\Tr \partial_\nu u_{q,r}\|_{W^{1-\frac{1}{p_i}, p_i}(\partial M)}+\| u_{q,r} \|_{W^{0, p_i}(M)}
\end{array}
$$
Then, doing more straightforward calculations we get
$$
\begin{array}{ccl}
    C_{p_i}\|u_{q,r}\|_{W^{2, p_i}(M)} & \leq & \frac{n-2}{4(n-1)}\| R'u_{q,r}^{q-1}-R u_{q,r}\|_{L^{p_i}(M)}+\frac{n-2}{2}\|H'(\Tr u_{q,r})^{r-1}-H\Tr u_{q,r}\|_{W^{1-\frac{1}{p_i}, p_i}(\partial M)}\\
    &&+\| u_{q,r} \|_{L^{ p_i}(M)}
\end{array}
$$
because $u_{q,r}$ is a solution to equation (\ref{e:F_subcritical_problem}). Hence 
$$
\begin{array}{ccl}
    C_{p_i}\|u_{q,r}\|_{W^{2, p_i}(M)} & \leq &\frac{n-2}{4(n-1)}\max |R'|\| u_{q,r}^{q-1}\|_{L^{ p_i}(M)}+\left(\frac{n-2}{4(n-1)}\max |R|+1\right)\| u_{q,r}\|_{L^{p_i}(M)}\\  
    &&+\frac{n-2}{4(n-1)}\left(\|H'(\gamma u_{q,r})^{r-1}\|_{W^{1-\frac{1}{p_i}, p_i}(\partial M)}+\|H\gamma u_{q,r}\|_{W^{1-\frac{1}{p_i}, p_i}(\partial M)}\right) 
\end{array}
$$
and by corollary A.5 of \cite{HT13} there is $D_1>0$ such that, if $D_2>\max\{\|H'\|_{W^{s,t}(\partial M)}, \|H\|_{W^{s,t}(\partial M)}\}$,
$$
\begin{array}{ccl}
    C_{p_i}\|u_{q,r}\|_{W^{2, p_i}(M)}&\leq& \frac{n-2}{4(n-1)}\max |R'|\| u_{q,r}^{q-1}\|_{L^{ p_i}(M)}+\left(\frac{n-2}{4(n-1)}\max |R|+1\right)\| u_{q,r}\|_{L^{p_i}(M)}\\  
    &&+\frac{n-2}{4(n-1)}D_1D_2\left(\|(\gamma u_{q,r})^{r-1}\|_{W^{1-\frac{1}{p_i}, p_i}(\partial M)}+\|\gamma u_{q,r}\|_{W^{1-\frac{1}{p_i}, p_i}(\partial M)}\right)\\
    &\leq&\frac{n-2}{4(n-1)}\max |R'|\| u_{q,r}^{q-1}\|_{L^{ p_i}(M)}+\left(\frac{n-2}{4(n-1)}\max |R|+1\right)\| u_{q,r}\|_{L^{p_i}(M)}\\  
    &&+\frac{n-2}{4(n-1)}D_1D_2D_3\left(\| u_{q,r}^{r-1}\|_{W^{1, p_i}(M)}+\| u_{q,r}\|_{W^{1, p_i}(M)}\right), 
\end{array}
$$
with $D_3$ given by the trace inequality.

Now, of course $\|u_{q,r}^{q-1}\|_{L^{p_i}(M)}=\|u_{q,r}\|^{q-1}_{L^{p_i(q-1)}(M)}$. We can find a similar estimate for $\|u^{r-1}_{q,r}\|_{W^{1,p_i}(M)}$.

Using the fact that $\nabla \left(u^{r-1}\right)=(r-1)u^{r-2}\nabla u$, for any $u$:
\begin{equation*}
\begin{array}{ccl}
   \|u^{r-1}\|^{p_i}_{W^{1,p_i}(M)}  &=& \int_M |u|^{p_i(r-2)}(|u|^{p_i}+r|\nabla u|^{p_i}) dV_g \\
    &\leq & r\int_M |u|^{p_i(r-2)}(|u|^{p_i}+|\nabla u|^{p_i}) dV_g,
\end{array}
\end{equation*}
because $r>1$. Combining this with H\"older's inequality we find that
\begin{equation}
    \|u^{r-1}\|^{p_i}_{W^{1,p_i}(M)}\leq r\|u^{p_i(r-2)}\|_{L^{\frac{2\bar q-1}{\bar q-1}}(M)}\||u|^{p_i}+|\nabla u|^{p_i}\|_{L^{\frac{2\bar q-1}{\bar q}}(M)}.
\end{equation}
So, since $r\leq \bar q+1$, there is $D_4>0$ such that
\begin{equation}
    \|u^{r-1}\|_{W^{1,p_i}(M)}\leq D_4\|u\|^{r-2}_{L^{p_i(2\bar q-1)}(M)}\|u\|_{W^{1,\frac{p_i(2\bar q-1)}{\bar q}}(M)}.
\end{equation}
Now, coming back to the main flow of the proof, we have that
\begin{equation}
\begin{array}{ccl}
    \|u_{q,r}\|_{W^{2, p_i}(M)} & \leq & C^{(1)}_{p_{i+1}}\| u_{q,r}\|^{q-1}_{L^{ p_i(2\bar q-1)}(M)}+C^{(2)}_{p_{i+1}}\|u_{q,r}\|_{L^{p_i}(M)}+C^{(3)}_{p_{i+1}}\|u_{q,r}\|_{W^{1,p_i}(M)}\\
     & &+C^{(4)}_{p_{i+1}}\|u_{q,r}\|^{r-2}_{L^{p_i(2\bar q-1)}(M)}\|u_{q,r}\|_{W^{1,\frac{p_i(2\bar q-1)}{\bar q}}(M)},
\end{array}
\end{equation}
if $C^{(1)}_{p_{i+1}}=\frac{n-2}{4(n-1)}\max |R'|[V_g(M)]^{\frac{q-1}{p_i(q-1)}-\frac{q-1}{p_i(2\bar q-1)}}$, $C^{(2)}_{p_{i+1}}=\left(\frac{n-2}{4(n-1)}\max |R|+1\right)$, $C^{(3)}_{p_{i+1}}=\frac{n-2}{4(n-1)}D_1D_2D_3$ and $C^{(4)}_{p_{i+1}}=\frac{n-2}{4(n-1)}D_1D_2D_3D_4$. Now, if $\| u_{q,r}\|_{L^{ p_i(2\bar q-1)}(M)}\leq 1$, its corresponding terms are controlled. If not
\begin{equation}
\begin{array}{ccl}
    \|u_{q,r}\|_{W^{2, p_i}(M)} & \leq & C^{(1)}_{p_{i+1}}\| u_{q,r}\|^{2\bar q-1}_{L^{ p_i(2\bar q-1)}(M)}+C^{(2)}_{p_{i+1}}\|u_{q,r}\|_{L^{p_i}(M)}+C^{(3)}_{p_{i+1}}\|u_{q,r}\|_{W^{1,p_i}(M)}\\
     & &+C^{(4)}_{p_{i+1}}\|u_{q,r}\|^{\bar q-1}_{L^{p_i(2\bar q-1)}(M)}\|u_{q,r}\|_{W^{1,\frac{p_i(2\bar q-1)}{\bar q}}(M)}.
\end{array}
\end{equation}
Finally, by Sobolev embedding theorems, there are $K^{(1)}_{i+1},\  K^{(2)}_{i+1},\  K^{(3)}_{i+1}>0\  K^{(4)}_{i+1}$ such that
\begin{equation}
\begin{array}{ccl}
    \|u_{q,r}\|_{W^{2, p_i}(M)} &\leq& K^{(1)}_{i+1}\|u_{q,r}\|^{2\bar q-1}_{W^{2, p^{(1)}_{i+1}}(M)}+K^{(2)}_{i+1}\|u_{q, r}\|_{W^{2, p^{(2)}_{i+1}}(M)}+K^{(3)}_{i+1}\|u_{q,r}\|_{W^{2,p^{(3)}_{i+1}}(M)} \\
     & &+K^{(4)}_{i+1}\|u_{q,r}\|^{\bar q-1}_{W^{2, p^{(1)}_{i+1}}(M)}\|u_{q,r}\|_{W^{2,p^{(4)}_{i+1}}(M)}
\end{array}
\end{equation}
with $p^{(1)}_{i+1}=\frac{np_i(2\bar q-1)}{n+2p_i(2\bar q-1)}\geq p^{(2)}_{i+1}$ and $p^{(4)}_{i+1}=\frac{np_i(2\bar q-1)}{n\bar q+p_i(2\bar q-1)}\geq p^{(3)}_{i+1}$. But we also have $p^{(1)}_{i+1}>p^{(4)}_{i+1}$! So, since $Q>2\bar q$, the usual bootstrap arguments can be used to give a constant $D>0$, independent of $q$, $r$ subcritical such that
\begin{equation}
    \|u_{q,r}\|_{W^{2,\frac{n}{2}}(M)}\leq D\|u_{q,r}\|_{L^Q(M)}
\end{equation}
and Lemma \ref{l:extra_regularity} says there is $C>0$, independent of $q,\ r$ such that
\begin{equation}
    \|u_{q,r}\|_{W^{2, \frac{n}{2}}(M)}\leq C.
\end{equation}
The bound on the traces come straight from the trace inequality.
\end{proof}

We can finally prove our main theorem

\begin{theorem}
\label{t:result_yamabe_negative}
Let $(M, \partial M)$ be a smooth, Yamabe-negative manifold with boundary, $g$ a $W^{s, p}$ Riemannian metric, $sp>n$, $s\geq 1$ with bounded scalar curvature $R$ and mean extrinsic curvature $H\in W^{r,t}(\partial M)$ along the boundary, $t>1$, $r\geq 1$ and $rt>n-1$. If $R'$, is a bounded non-positive function, $H'\in W^{r,t}(\partial M)$, then there is a function $u\in W^{2,\frac{n}{2}}(M)\cap L^{\infty}(\partial M)$ such that the metric $g'=u^{2\bar q-2}g$ has scalar curvature $R'$ inside $M$ and mean extrinsic curvature $H'$ along the boundary $\partial M$ if, and only if, $(Z,Z_\partial )$ is Yamabe positive.
\end{theorem}

\begin{proof}
First, since $(M, \partial M)$ is Yamabe-negative, equation (\ref{eq:constant_sign_curvatures}) and theorem \ref{t:subcrit} give $\phi>0$ a function defined over $\bar M$ such if $\tilde g=\phi^{2\bar q-2}g$, $R_{\tilde g}, H_{\tilde g}<0$ everywhere and condition (\ref{eq:condition_negativity_of_curvatures}) is satisfied by $\tilde g$. So we proceed with $(M,\partial M, \tilde g)$ instead of $(M,\partial M, g)$.

Now, if $(Z,Z_\partial)$ is Yamabe positive, let $(q_n, r_n)$ be a sequence, increasing on each coordinate, converging to $(2\bar q, \bar q+1)$. Then we have a sequence $\{u_{q_n, r_n}\}_n$ of minimizers of $F_{q_n, r_n}$. Now, this sequence is bounded in $W^{2, \frac{n}{2}}(M)$, so there is $u\in W^{2, \frac{n}{2}}(M)$ such that, up to subsequences, $u_{q_n, r_n}\to u$ in $W^{1,2}(M)$ and $u_{q_n, r_n}\to u$ uniformly in $\bar M$. So $u$ is a solution to
\begin{equation}
\label{e:critical_problem}
\begin{cases}
    -\Delta u+\frac{n-2}{4(n-1)}R_{\tilde g} u=\frac{n-2}{4(n-1)}R' u^{2\bar q-1},\ in\ M\\
\Tr\partial_\nu u+\frac{n-2}{2}H_{\tilde g}\Tr u=\frac{n-2}{2}H'(\Tr u)^{\bar q},\ in\ \partial M
\end{cases}.
\end{equation}
On the other hand, assume there is such solution $u$. Then we can calculate the Yamabe invariant of $(Z, Z_\partial)$ with respect to the metric $g'=u^{2\bar q-2}\tilde g$. If $B^{2,2}_1(Z, Z_\partial)$ is empty, $\mathcal Y^{2,2}_1(Z,Z_\partial)=+\infty$ by definition. Otherwise, there is a minimizer $\bar u\in B^{2,2}_1(Z, Z_\partial)$ such that $\mathcal{Y}^{2,2}_1(Z,Z_\partial)=E(\bar u)$.

But if $\bar u\in W^{1,2}(Z, Z_\partial)$, $\bar u$ is supported in $Z\cup Z_\partial$, so
$$
\int_Z R'\bar u^2dV_g=\int_{Z_\partial} H'(\Tr\bar u)^2 d\sigma_g=0
$$
and $E(\bar u)=\int_Z|\nabla \bar u|^2dV_g\geq 0$. We just have to rule out $E(u)=0$.

If $Z$ has positive measure, $E(\bar u)=0$ if, and only if, $\bar u$ is a constant. But $\bar u$ is supported in $(Z,Z_\partial)$, so either  $\bar u\equiv 0$, which is not possible because $\bar u\in B^{2,2}_1(Z,Z_\partial)$ or $(Z,Z_\partial)=(M,\partial M)$ and the manifold is Yamabe zero, which is not the case.

Now, if $Z$ has measure zero, and $\bar u$ is a function in $W^{1,2}(M)$ supported in $Z$, $\Tr \bar u\equiv 0$. So $B^{2,2}_1(Z,Z_\partial)=\emptyset$, which is not the case for us here. So $E(\bar u)>0$.
\end{proof}


\section{Examples and Consequences}
\label{sec:examples_and_consequences}

With our main result stated, it seems worthwhile to study some specific cases in order to make a better sense of the tools we just developed.

In the next two corollaries we show some cases where we can guarantee, under weaker conditions, that $\mathcal{Y}(Z,Z_\partial)$ is positive. Behind both results is the fact that test functions in $W^{1,2}(\Omega, \Sigma)$ in general cannot be concentrated only on the boundary. That is, if $u\in W^{1,2}(\Omega,\Sigma)$ is non-zero in a point $p\in \Sigma$ in the boundary, it must be non-zero in nearby points inside the manifold. But in that case, $u\equiv 0$ outside $\Omega\cup \Sigma$, so that is only possible if $p$ lies in the closure of $\Omega$. As a result we can see that, in fact, only the set $\bar \Omega$ counts towards the evaluation of $\Yam(\Omega,\Sigma)$.
\begin{corollary}
In the context of Theorem \ref{t:result_yamabe_negative}, if $Z_\partial\cap \bar Z=\emptyset$, one can find the metric $g'$ if, and only if, $\mathcal{Y}(Z,\emptyset)$ (which equals $\Yam(Z)$ in the sense of \cite{DM15}) is positive.
\begin{proof}
In this case, $u\in W^{1,2}(Z,Z_\partial)$ implies $(\Tr u)\vert_{Z_\partial}\equiv 0$.
\end{proof}
\end{corollary}

Taking the same reasoning to a more extreme consequence, if $\Omega\cup \Sigma$ has no interior in $\bar M$, $W^{1,2}(\Omega,\Sigma)$ is trivial and $B^{q,r}_b(\Omega,\Sigma)$ is empty regardless of $q$, $r$ or $b$. As a consequence, if $\Omega$ has measure zero, $\Yam(\Omega,\Sigma)=+\infty$.
\begin{corollary}
In the context of Theorem \ref{t:result_yamabe_negative}, if $Z$ has measure zero, one can find the metric $g'$ regardless of $H'$, as long as $H'\leq 0$.
\end{corollary}

\subsection{Open Sets in $\mathbb{R^n}$}

A set of cases for which we can calculate the Yamabe invariant - and compare with Escobar results in \cite{Esco92b} for example - is when M is an open bounded subset of $R^n$. In that case, $R\equiv 0$ and the energy reduces to
\begin{equation}
    E(u)=\int_{\Omega}\|\nabla u\|^2dV_g+\frac{n-2}{2}\int_\Sigma H(\Tr u)^2d\sigma_g.
\end{equation}
In particular, if $\Sigma=\emptyset$, $E(u)\geq 0$ for all $u\in W^{1,2}(\Omega,\emptyset)$, with equality if, and only if, $u\equiv 0$, since there are no other constants in $W^{1,2}(\Omega, \emptyset)$. As a consequence, all open sets in the interior of flat manifolds are Yamabe positive. In particular, $\mathcal{Y}(M,\emptyset)>0$.

Our Theorem \ref{t:result_yamabe_negative} then implies that if $H'<0$ everywhere, $R'\leq 0$, $M$ is a bounded open subset of $\mathbb{R}^n$ and $\Yam(M,\partial M)<0$, there is a metric conformal to the Euclidean metric on $(M,\partial M)$ with $R'$ as its scalar curvature inside the manifold and $H'$ its mean curvature over $\partial M$. With the appropriate hypothesis, this implies the negative case for Theorem 1 in \cite{Esco92b}. We do not have restrictions on the dimension for the Yamabe-negative case though.

Escobar's classification also shows that the Yamabe-negative condition cannot be dropped in Theorem \ref{t:result_yamabe_negative}, since it is {\textit{not}} true that one can realize $R'\equiv 0$ and $H'<0$ in Yamabe-positive manifolds with boundary that are subsets of $\mathbb{R}^n$. For example, a straightforward application of the divergence theorem shows there is no solution to
\begin{equation}
    \begin{cases}
    -\Delta u \equiv 0,\ in\ \mathbb{D}^n,\\
    \Tr \partial_\nu u+\frac{n-2}{2}\Tr u=-\frac{n-2}{2}(\Tr u)^{\bar q},\ in\ \mathbb{S}^{n-1},
    \end{cases}
\end{equation}
what does not contradict our theorem because the unit ball with its boundary is Yamabe-positive, since for $u\in W^{1,2}(\mathbb{D}^n, \mathbb{S}^{n-1})$
\begin{equation}
    E(u)=\int_{\mathbb{D}^n}\|\nabla u\|^2dV_g+\int_{\mathbb{S}^{n-1}}(\Tr u)^2d\sigma_g
\end{equation}
which is non-negative and vanishes if, and only if, $u\equiv 0$.

\subsection{Lichnerowicz Equation}

The Lichnerowicz equation arises in General Relativity in the formulation of the initial value problem for Einstein equations. Briefly, one wants to find a metric in a given conformal class that satisfies the geometric constraint equations for a spacetime that solves Einstein's equations with given matter fields. The Lichnerowicz equation then is a scalar condition on the conformal factor. A derivation can be found, for example, in section 1.3 of \cite{Hebey14}.

In \cite{HT13} the authors discuss the Lichnerowicz equation in the context of manifolds with boundary and find existence conditions to a wide class of problems that can be stated as boundary value formulations of the Lichnerowicz equation. In their theorem 6.2, the authors exactly deal with Yamabe negative manifolds with boundary and prove that, in that case, the existence of solutions to the Lichnerowicz equation corresponds to the existence of solutions to a prescribed curvature problem. Our Theorem \ref{t:result_yamabe_negative}, then, solves the problem to a class of those equations.

To state it precisely, consider the Lichnerowicz equation
\begin{equation}
    \label{e:Lichnerowicz}
    \begin{cases}
    -\Delta u+\frac{n-2}{4(n-1)}Ru=\frac{n-2}{4(n-1)}R'u^{2\bar q-1}+a_{w}u^{-2\bar q-1},\ in\ M\\
    \Tr \partial_\nu u+\frac{n-2}{2}H\Tr u=\frac{n-2}{2}H'(\Tr u)^{\bar q}-b_wu^{-\bar q},\ in\ \partial M,
    \end{cases}
\end{equation}
notice that there is no part of the boundary where the initial condition specified is a Dirichlet-type condition. In that case, \cite{HT13} shows that, if $(M,\partial M)$ is Yamabe negative, the existence of solutions to (\ref{e:Lichnerowicz}) is equivalent to the existence of solutions to (\ref{e:critical_problem}) if $a_w\geq 0$, $b_w\leq 0$, $R'\leq 0$ and $H'\leq 0$, exactly the case studied in this paper. So the conditions on Theorem \ref{t:result_yamabe_negative} are also the conditions on the existence of solutions to (\ref{e:Lichnerowicz}), provided $a_w\geq 0$ and $b_w\leq 0$.


\printbibliography

\end{document}